\documentclass[a4paper,3p]{elsarticle}

\usepackage[english]{babel}
\usepackage{amsmath,amsthm}
\usepackage{amsfonts}
\usepackage{url}
\usepackage{amssymb}
\usepackage{psfig}
\usepackage[%
breaklinks%
,colorlinks%
,linkcolor=red
,anchorcolor=blue%
,pagecolor=blue%
,citecolor=blue%
,bookmarks=false%
]{hyperref}

\theoremstyle{Definition}

\numberwithin{equation}{section}

\newtheorem{theorem}{Theorem}
\newtheorem{corollary}{Corollary}
\newtheorem{lemma}{Lemma}
\newtheorem{remark}{Remark}
\newtheorem{definition}{Definition}
\newtheorem{example}{Example}

\def\para{\vspace{1.5mm}}

\def\gcd{{\rm gcd}}
\def\lcm{{\rm lcm}}

\def\Resultant{{\rm Res}}

\def\ord{{\rm ord}}

\def\deg{{\rm deg}}

\def\deg{{\rm deg}}

\def\ox{\,{\overline x}\,}

\def\resultant{{\rm resultant}}

\begin{document}
\begin{frontmatter}
 \title{Infinity Branches and Asymptotic Analysis of Algebraic Space Curves: New Techniques and Applications}
\author[alcalar]{Sonia P\'erez-D\'iaz}
	\ead{sonia.perez@uah.es}
\author[ucas]{Li-Yong Shen\corref{cor}}
	\ead{lyshen@ucas.ac.cn}
\author[ucas]{Xin-Yu Wang}
	\ead{wangxinyu205@mails.ucas.ac.cn}
\author[valen]{R. Magdalena-Benedicto }
	\ead{rafael.magdalena@uv.es}
\address[alcalar]{Departamento de F\'{\i}sica y Matem\'aticas, Universidad de Alcal\'a, E-28871 Madrid, Spain}
\address[ucas]{School of Mathematical Sciences, University of Chinese Academy of Sciences,100049, Beijing, China}
\address[valen]{ Department of Electronic Engineering, Universidad de Valencia, 46100  Valencia, Spain}

\begin{abstract}

Let \( \mathcal{C} \) represent an irreducible algebraic space curve defined by the real polynomials \( f_i(x_1, x_2, x_3) \) for \( i = 1, 2 \). It is a recognized fact that a birational relationship invariably exists between the points on \( \mathcal{C} \) and those on an associated irreducible plane curve, denoted as \( \mathcal{C}^p \). In this work, we leverage this established relationship to delineate the asymptotic behavior of \( \mathcal{C} \) by examining the asymptotes of \( \mathcal{C}_p \). Building on this foundation, we introduce a novel and practical algorithm designed to efficiently compute the asymptotes of~\( \mathcal{C} \), given that the asymptotes of \( \mathcal{C}_p \) have been ascertained.

\end{abstract}


\end{frontmatter}

\section{Introduction}
In \cite{paper1}, we introduced the concepts of {\it infinity branches} and {\it approaching curves}. Intuitively speaking, {\it the infinity branches} reflect the status of a given curve
at the points with ``sufficiently large coordinates''. From this notion, we get the concept of {\it convergent branches} and {\it  approaching curves}. More precisely, we say that {\it two infinity branches converge} 
if they get closer as they tend to infinity and {\it  a curve  $\overline{\cal C}$ approaches ${\cal C}$} at its
infinity branch $B$ if $\overline{\cal C}$ has an infinity branch convergent with $B$.  In fact, if $\overline{\cal C}$ approaches ${\cal C}$ at all of its infinity
branches and reciprocally, we say that both curves have the same {\it asymptotic behavior}. 
 \para 
 
From these notions, important properties were derived and in fact, these concepts led to the development of an algorithm for comparing the behavior of two implicitly defined algebraic plane curves at infinity. Specifically, we characterized the finiteness of the {\it Hausdorff distance} between two algebraic curves in $n$-dimensional space, which is related to the asymptotic behavior of the curves (see \cite{paper1}  and \cite{Hausdorff}).

\para

Building upon the concepts and results in \cite{paper1}, in \cite{BlascoPerezII}, we addressed the computation of {\it asymptotes} for the infinity branches of a given implicitly defined algebraic curve $\mathcal{C}$. The {\it asymptotes of an infinity branch} of $\mathcal{C}$ describe the branch's behavior at points with sufficiently large coordinates. While traditionally asymptotes of a curve are lines where the distance between the curve and the line approaches zero at infinity, we demonstrated that an algebraic plane curve may have more general curves, rather than just lines, describing its behavior at infinity. Consequently, in \cite{BlascoPerezII}, we developed an algorithm for computing {\it generalized asymptotes} (or {\it g-asymptotes}) using Puiseux series, and we presented important properties related to this new concept. Some improvements of this algorithm for the implicit case of a plane curve is presented in \cite{NewImpl-As}, and also the parametric case is studied in \cite{NewPara-As}.

\para

The applicability of these results  is crucial in computer-aided geometric design (CAGD); see e.g. \cite{farin},  \cite{farin2},  \cite{hoffmann}, or  \cite{hoschek}. These results provide new concepts and computational techniques that yield insights into the behavior of algebraic curves at infinity. For example, the infinity branches of an implicit plane curve $\mathcal{C}$ are essential for studying the topology of $\mathcal{C}$ (\cite{Gao},\cite{lalo}, \cite{Hong}), graph sketching  or even for detecting its symmetries (see e.g. \cite{Juangetopologia}, \cite{Juangesimetrias}). Additionally, the results obtained play a significant role in approximate parametrization problems or in analyzing the {\it Hausdorff distance} between two curves, which is important for measuring the closeness between them (\cite{Bai}, \cite{Bizzarri}, \cite{Chen}, \cite{PSS1}, \cite{PRSS}). In general, CAGD serves as a natural environment for practical applications involving algebraic curves and surfaces. Particularly, the results and methods presented in this paper pave the way for new avenues in studying the behavior of algebraic space curves, with anticipated extensions to higher dimensions and the case of surfaces (\cite{Surf-As}).

\para

Additionally, the investigation of intersection curves between two surfaces is a cornerstone challenge extensively explored in CAGD (see \cite{Bar} and \cite{hoschek}), finding broad application in CAD/CAM and manufacturing. Diverse methods for computing these curves, falling into numerical and algebraic categories (see    \cite{Heo} and \cite{LiyongChenJia}), have been developed. However, several early studies, such as \cite{Tu}, have explored the intersection curves of two quadrics in the spaces ${\Bbb P}{\Bbb C}^3$ or ${\Bbb P}{\Bbb R}^3$, where the behavior of the space curve at infinity must be considered. Nevertheless, existing approaches still do not fully address the challenges posed by infinite branches within intersections.  Analyzing infinite branches poses a unique challenge, distinct from cases involving bounded intersections, and demands a more sophisticated approach. This prompts us to delve deeper into understanding infinite intersection curves. This paper is expected to helpful in  tackling critical yet underexplored aspect of the infinite part of surface intersection curves.

\para

Motivated by the results and applications mentioned above, we sought to generalize the foundations and methods presented in \cite{BlascoPerezII} and \cite{NewImpl-As} to the case of space curves. More precisely, in this paper, we consider an irreducible real algebraic space curve $\mathcal{C}$ implicitly defined by two irreducible real polynomials over the complex numbers field $\mathbb{C}$. We address the problem of computing the asymptotes of the infinity branches of $\mathcal{C}$ in the most efficient way. For this purpose, we show that the asymptotes of $\mathcal{C}$ can be obtained from the asymptotes of ${\mathcal{C}_p}$, where ${\mathcal{C}_p}$ is a planar curve birationally equivalent to the given spatial curve $\mathcal{C}$. Thus, the problem reduces to demonstrating the equivalence between the asymptotes of these curves and providing an effective algorithm for computing the asymptotes of $\mathcal{C}$ once the asymptotes of ${\mathcal{C}_p}$ are determined. It is easy to note that this approach can be easily extended to  curves in $n$-dimensional space.

\para

The structure of the paper unfolds as follows: in Section \ref{S-notation}, we introduce the notation and  some previous results as needed.
In Section \ref{S-computationasymptote}, we introduce the concepts of perfect curve and generalized asymptote (or g-asymptote).
Finally, we conclude the paper in Section \ref{Sect4}, summarizing the obtained results, highlighting the new contributions of this paper, and proposing topics for further study.

 \section{Notation and terminology}\label{S-notation}

 In this section, we introduce key concepts and terminology essential for understanding the subsequent content of the paper. We start by reviewing previous findings related to Puiseux series (\cite{Alon92}, \cite{paper1}, \cite{Duval89}, \cite{SWP},  and \cite{Walker}), and {\it infinity branches} and {\it approaching curves} (these notions are originally proposed in \cite{paper1}). Subsequently, we introduce important results and tools derived from these notions.

\para

Let ${\Bbb C}[[t]]$ denote the domain of formal power series in the indeterminate $t$ with coefficients in the field ${\Bbb C}$, represented as the set of all sums of the form $\sum_{i=0}^\infty a_it^i$, where $a_i \in {\Bbb C}$. The field of formal Laurent series, denoted as ${\Bbb C}((t))$, is the quotient field of ${\Bbb C}[[t]]$. It's well established that every nonzero formal Laurent series $A \in {\Bbb C}((t))$ can be expressed in the form $A(t) = t^k\cdot(a_0+a_1t+a_2t^2+\cdots)$, where $a_0 \neq 0$ and $k\in \Bbb Z$. Moreover, the field ${{\Bbb C}\ll t\gg}:= \bigcup_{n=1}^\infty {\Bbb C}((t^{1/n}))$ is termed the field of formal Puiseux series. Puiseux series are power series of the form $\varphi(t)=m+a_1t^{N_1/N}+a_2t^{N_2/N}+a_3t^{N_3/N} + \cdots\in{\Bbb C}\ll t\gg$, where $a_i\neq0$ for all $i\in {\Bbb N}$, $N, N_i\in {\Bbb N}$ with $i\geq 1$, and $0<N_1<N_2<\cdots$. The natural number $N$ is referred to as the ramification index of the series, denoted as $\nu(\varphi)$ (see \cite{Duval89}). The order of a nonzero Puiseux or Laurent series $\varphi$ is defined as the smallest exponent of a term with a non-vanishing coefficient in $\varphi$, denoted as $\ord(\varphi)$.

A fundamental property of Puiseux series is encapsulated in Puiseux's Theorem, which asserts that if $\Bbb K$ is an algebraically closed field, then the field ${\Bbb K}\ll x\gg$ is likewise algebraically closed (refer to Theorems 2.77 and 2.78 in \cite{SWP}). A constructive proof of Puiseux's Theorem is offered through the Newton Polygon Method (see, for instance, Section 2.5 in \cite{SWP}).

\para

In the following, we introduce the concept of the infinity branch of a space curve, which constitutes a crucial tool for the subsequent developments in this paper. Consider an irreducible space curve ${\cal C}\subset {\Bbb C}^3$ defined by two polynomials $f_i(x_1,x_2,x_3) \in {\Bbb R}[x_1,x_2,x_3]$ for $i=1,2$.  That is, in  this paper real algebraic space curves are considered implicitly defined as the intersection of two surfaces.

\para

The infinity branches of a curve intuitively represent the regions of the curve extending to infinity, corresponding to the infinity places of the corresponding projective curve (see Section 2.5 in \cite{SWP}). In \cite{paper1} (Section 3), we define these branches for a given plane curve $\cal C$ as sets of the form $B=\{(z,r_j(z))\in {\Bbb C}^2: \,z\in {\Bbb C},\,|z|>M\}\subset{\cal C},\,M\in {\Bbb R}^+$, where $r_{j}(z)$ for $j=1,\ldots,N$ are Puiseux series (\cite{Duval89}). Specifically, $f(z,r_j(z))=0$ for sufficiently large values of $z$, where $f(x_1, x_2)$ denotes the polynomial  defining  $\cal C$.

\para

In the following, we extend the concept of infinity branches to  space curves and we provide a mathematical framework for these entities.  Additionally, we  see how there is actually a one-to-one correspondence between the branches of a spatial curve and a plane curve. This idea is of paramount importance because it allows us to reduce the problem by lowering the dimension and thus generalize the problem we are dealing with, to higher dimensions.

\para

It's worth noting that our work is conducted over $\Bbb C$, although we assume that the curve possesses infinitely many points in the affine plane over $\Bbb R$. Consequently, $\cal C$ is endowed with real defining polynomials (see  Chapter 7 in \cite{SWP}). This assumption of reality is intrinsic to the problem at hand, but the theory developed in this paper can be similarly applied to the case of complex non-real curves.

\para

Let ${\cal C}^*$ denote the corresponding projective curve defined by the homogeneous polynomials $F_i(\ox) \in {\Bbb R}[\ox],\,\ox=(x_1,x_2,x_3,x_4)$ for $i=1,2$. Moreover, consider a point $P=(1:m_2:m_3:0)$, where $m_2,m_3\in {\Bbb C}$, lying at infinity on ${\cal C}^*$. Additionally, we examine the curve implicitly defined by the polynomials $g_i(x_2,x_3,x_4):=F_i(1,x_2,x_3,x_4)\in {\Bbb R}[x_2,x_3,x_4]$ for $i=1,2$. Notably, $g_i(p)=0$ where $p=(m_2,m_3,0)$. Let $I\in {\Bbb R}(x_4)[x_2,x_3]$ represent the ideal generated by $g_i(x_2,x_3,x_4)$ for $i=1,2$ in the ring ${\Bbb R}(x_4)[x_2,x_3]$. As ${\cal C}$ is not contained in a hyperplane $x_4 =c$ for $c\in {\Bbb C}$, we infer that $x_4$ is not algebraic over $\Bbb R$. Under this assumption, the ideal $I$ (i.e., the system of equations $g_1 =  g_2 = 0$) has only finitely many solutions in the three-dimensional affine space over the algebraic closure of ${\Bbb R}(x_4)$ (which is contained in ${{\Bbb C}\ll x_4\gg}$). Hence, there exist finitely many pairs of Puiseux series $(\varphi_{2}(t),\varphi_{3}(t)) \in {{\Bbb C}\ll t\gg}^2$ such that $g_i(\varphi_{2}(t),\varphi_{3}(t),t) = 0$ for $i=1,2$, and $\varphi_{k}(0)=m_k$ for $k=2,3$. Each pair $(\varphi_{2}(t),\varphi_{3}(t))$ serves as a solution of the system associated with the infinity point $(1:m_2:m_3:0)$, and $\varphi_{2}(t)$ and $\varphi_{3}(t)$ converge in a neighborhood of $t= 0$. Furthermore, since $\varphi_{k}(0)=m_k$ for $k=2,3$, these series lack terms with negative exponents; specifically, they take the form
$$\varphi_{k}(t)=m_k+\sum_{i\geq 1}a_{i,k} t^{N_{i,k}/N_k},$$
where $N_k,N_{i,k}\in {\Bbb N}$, $0<N_{1,k}<N_{2,k}<\cdots$.

\para

It's noteworthy that if $\varphi(t):=(\varphi_{2}(t),\varphi_{3}(t))$ constitutes a solution of the system, then $\sigma_{\epsilon}(\varphi)(t):=(\sigma_{\epsilon}(\varphi_{2})(t),\sigma_{\epsilon}(\varphi_{3})(t))$ represents another solution of the system. Here,
$\sigma_{\epsilon}(\varphi_{k})(t)=m_k+\sum_{i\geq 1}a_{i,k}\epsilon^{\lambda_{i,k}} t^{N_{i,k}/N_k},$
$N:=\lcm(N_2,N_3)$, $\lambda_{i,k}:=N_{i,k}N/N_k\in {\Bbb N}$, and $\epsilon^N=1$ (see Section 5.2 in \cite{Alon92}). These solutions are referred to as the conjugates of $\varphi$. The set of all distinct conjugates of $\varphi$ is termed the conjugacy class of $\varphi$, and the number of different conjugates of $\varphi$ is $N=\nu(\varphi)$. In the following, let $\varphi$ be any representant of these series.

\para

Under these conditions and following the reasoning in \cite{paper1} (see Section 3), we establish that there exists $M \in {\Bbb R}^+$ such that for $i\in \{1,2\}$,
$$F_i(1:\varphi_{2}(t):\varphi_{3}(t):t)=g_i(\varphi_{2}(t),\varphi_{3}(t),t)=0,\quad \mbox{for $t\in {\Bbb C}$ and $|t|<M$}.$$
This implies that $F_i(t^{-1}:t^{-1}\varphi_{2}(t):t^{-1}\varphi_{3}(t):1)=f_i(t^{-1},t^{-1}\varphi_{2}(t),t^{-1}\varphi_{3}(t))=0,$
for $t\in {\Bbb C}$ and $0<|t|<M$.

\para

Now, setting $t^{-1}=z$, we deduce that for $i\in \{1,2\}$,
$$f_i(z,r_{2}(z),r_{3}(z))=0,\quad \mbox{for $z\in {\Bbb C}$ and $|z|>M^{-1}$, where}$$
\begin{equation}\label{Eq-conjugates}
r_{k}(z)=z\varphi_{k}(z^{-1})=m_kz+a_{1,k}z^{1-N_{1,k}/N_k}+a_{2,k}z^{1-N_{2,k}/N_k}+a_{3,k}z^{1-N_{3,k}/N_k} + \cdots,
\end{equation}
$a_{j,k}\neq 0$, $N_k,N_{j,k}\in {\Bbb N}$, $j=1,\ldots$, and $0<N_{1,k}<N_{2,k}<\cdots$.

\para

  With this, we introduce the concept of infinity branches. The subsequent definitions and results generalize those presented in Sections 3 and 4 in \cite{paper1} for algebraic plane curves.

\para

\begin{definition}\label{D-infinitybranch}
An infinity branch of a space curve ${\cal C}$ associated with the point at infinity $P=(1:m_2:m_3:0)$,  $m_2, m_3 \in {\Bbb C}$, is a set $ B=\{(z,r_{2}(z),r_{3}(z))\in {\Bbb C}^3: \,z\in {\Bbb C},\,|z|>M\}$, $M\in {\Bbb R}^+$, and the series $r_{2}$ and $r_{3}$ are given by (\ref{Eq-conjugates}).\\
\end{definition}

\begin{remark}\label{R-definfinitybranch}
Using Definition \ref{D-infinitybranch}, we observe that the points of the infinity branch have the form $(z,r_{2}(z),r_{3}(z))$, and by construction, these points belong to the curve, implying that $f_i(z,r_{2}(z),r_{3}(z))=0,\,i=1,2,$ for every $z\in {\Bbb C}$ with $|z|>M$. Additionally, from (\ref{Eq-conjugates}) and considering that $1-N_{i,k}/N_k<1,\,i=1,\ldots$, we deduce that $\lim_{z\rightarrow\infty}r_{k}(z)/z=m_k$ for $k=2,3$. In other words, $\lim_{z\rightarrow\infty}(1:r_{2}(z)/z:r_{3}(z)/z:1/z)=(1:m_2:m_3:0)$.
\end{remark}

\para
In the following, we assume without loss of generality that the given algebraic space curve ${\cal C}$ only has points at infinity of the form $(1 : m_2 : m_3 : 0)$ (otherwise, one consider a linear change of coordinates). More details on the  branches corresponding to different points of infinity are given in \cite{paper1}.

\para
Now, we introduce the notions of convergent branches and approaching curves. Intuitively speaking, two infinity branches converge if they get closer as they tend to infinity. This concept will allow us to analyze whether two space curves approach each other and it generalizes the notion introduced for the plane case (see Section 4 in \cite{paper1}). For this purpose, since we will be working over $\Bbb C$ or $\Bbb R$, in the following $d$ denotes the usual unitary or Euclidean distance (see Chapter 5 in \cite{Horn}).

\para

\begin{definition}\label{D-distance0}
Two infinity branches, $B=\{(z,r_2(z), r_3(z))\in {\Bbb C}^3:\,z\in {\Bbb C},\,|z|>M\}$ and $\overline{B}=\{(z,\overline{r}_2(z),$ $\overline{r}_3(z))\in {\Bbb C}^3:\,z\in {\Bbb C},\,|z|>\overline{M}\}\subset \overline{B}$ converge if  $\lim_{z\rightarrow\infty} d(({r}_2(z),{r}_3(z)), (\overline{r}_2(z),\overline{r}_3(z)))=0.$
\end{definition}

\para

\begin{remark} \label{R-distance0} From the previous definition, we note that two convergent infinity branches are associated with the same point at infinity (see Remark \ref{R-definfinitybranch}). Furthermore, two branches $B=\{(z,r_2(z), r_3(z))\in {\Bbb C}^3:\,z\in {\Bbb C},\,|z|>M\}$ and $\overline{B}=\{(z,\overline{r}_2(z),\overline{r}_3(z))\in {\Bbb C}^3:\,z\in {\Bbb C},\,|z|>\overline{M}\}$ are convergent if and only if the terms with non-negative exponents in the series $r_i(z)$ and $\overline{r}_i(z)$ are the same, for $i=2,3$.
\end{remark}

\para

In Definition \ref{D-distance1}, we introduce the notion of approaching curves  that is, curves that approach each other. For this purpose, we recall that given an algebraic space curve ${\cal C}$ over $\Bbb C$ and a point $p\in {\Bbb C}^3$, the distance from $p$ to ${\cal C}$ is defined as $d(p,{\cal C})=\min\{d(p,q):q\in{\cal C}\}$.

\para

\begin{definition}\label{D-distance1}
Let ${\cal C} $ be an algebraic space curve over ${\Bbb C}$ with an infinity branch $B$. We say that a curve ${\overline{{\cal C}}}$ approaches ${\cal C}$ at its infinity branch $B=\{(z,r_2(z), r_3(z))\in {\Bbb C}^3:\,z\in {\Bbb C},\,|z|>M\}\subset B$ if $\lim_{z\rightarrow\infty}d((z,r_2(z), r_3(z)),\overline{\cal C})=0$.
\end{definition}

\para

In the following, we state some important results concerning two curves that approach each other  (see Lemma 3.6, Theorem 4.11, Remark 4.12, and Corollary 4.13 in \cite{paper1}).

\para

\begin{theorem}\label{T-curvas-aprox}
Let ${\cal C}$ be an algebraic space curve over $\Bbb C$ with an infinity branch $B$. An algebraic space curve ${\overline{{\cal C}}}$ approaches ${\cal C}$ at $B$ if and only if ${\overline{{\cal C}}}$ has an infinity branch, $\overline{B}$, such that $B$ and $\overline{B}$ are convergent.
\end{theorem}

\para

 Note that ${\overline{{\cal C}}}$ approaches ${\cal C}$ at some infinity branch $B$ if and only if ${\cal C}$ approaches ${\overline{{\cal C}}}$ at some infinity branch $\overline{B}$. In the following, we say that ${\cal C}$ and ${\overline{{\cal C}}}$ {\it approach each other} or that they are  {\it approaching curves}. Two approaching curves have a common point at infinity.

\para

\begin{corollary}\label{C-approaching-curves}
Let $\cal C$ be an algebraic space curve with an infinity branch $B$. Let ${\overline{{\cal C}}}_1$ and ${\overline{{\cal C}}}_2$ be two different curves that approach $\cal C$ at $B$. Thus, ${\overline{{\cal C}}}_i$ has an infinity branch $\overline{B_i}$ that converges with $B$, for $i=1,2$. Furthermore, if  $\overline{B_1}$ and $\overline{B_2}$ are convergent then, ${\overline{{\cal C}}}_1$ and ${\overline{{\cal C}}}_2$ approach each other.
\end{corollary}

\subsection{Computation of infinity branches}\label{Sub-computationinfinitybranches}

Let $\mathcal{C}$ be an irreducible algebraic space curve defined by  $f_i(x_1,x_2,x_3)\in \mathbb{R}[x_1, x_2, x_3]$ for $i=1,2$. In this subsection, our focus lies on computing the infinity branches of $\mathcal{C}$. These branches are points of the form $(z,r_{2}(z),r_{3}(z))$, where $r_{k}(z)$ are conjugated Puiseux series (for $k=2,3$), such that $f_j(z,r_{2}(z),r_{3}(z))=0,\,j=1,2,$ for every $z\in \mathbb{C}$ with $|z|>M$ (see Definition \ref{D-infinitybranch} and Remark \ref{R-definfinitybranch}).

\para

To achieve this, and taking into account the previous reasoning, we find it necessary to compute a finite number of pairs of Puiseux series $(\varphi_{2}(t),\varphi_{3}(t)) \in {{\mathbb{C}}\ll t\gg}^2$ such that $g_i(\varphi_{2}(t),\varphi_{3}(t),t) = 0,\,i=1,2$ (where $r_{k}(z)=z\varphi_{k}(z^{-1})$). To tackle this problem, various methods could be employed. In \cite{paper1} (see Section 3), it is demonstrated how infinity branches for a given plane curve can be  computed using well-known implemented algorithms. Specifically, to compute the series expansions, the command \textsf{puiseux} included in the package \textsf{algcurves} of the computer algebra system \textsf{Maple} is employed (see Example 3.5).


\para

The approach in this section is rooted in the idea of reducing the problem of computing infinity branches for space curves to the planar case. In other words, we aim to compute the infinity branches of a given space curve $\mathcal{C}$ from those of a birationally equivalent plane curve, denoted as $\mathcal{C}^p$.

\para

To obtain $\mathcal{C}^p$, we may apply the method outlined in \cite{Bajaj} (see Sections 2 and 3). The curve obtained using this approach is birationally equivalent to $\mathcal{C}$, implying there exists a birational correspondence between the points of $\mathcal{C}^p$ and those of $\mathcal{C}$. In \cite{Bajaj}, it is established that $\mathcal{C}^p$ can always be obtained by projecting $\mathcal{C}$ along some {\it valid projection direction}. Once we have $\mathcal{C}^p$, we can compute its infinity branches using the procedure developed in \cite{paper1}. Finally, we utilize the aforementioned birational correspondence to obtain the infinity branches of $\mathcal{C}$ from those of $\mathcal{C}^p$.

\para

In what follows, we assume that the $x_3$-axis serves as a valid projection direction. Otherwise, we apply a linear change of coordinates (see Section 2 in \cite{Bajaj}). Let ${\cal C}^p$ denote the projection of $\cal C$ along the $x_3$-axis, and let $f^p(x_1,x_2)\in {\Bbb R}[x_1, x_2]$ be the implicit polynomial defining ${\cal C}^p$. In \cite{Bajaj} (see Section 3), a method is illustrated for constructing a birational mapping $h(x_1,x_2)=h_1(x_1,x_2)/h_2(x_1,x_2)$ such that $(x_1,x_2,x_3)\in{\cal C}$ if and only if $(x_1,x_2)\in{\cal C}^p$ and $x_3=h(x_1,x_2)$. To achieve this, one needs to compute a polynomial remainder sequence (PRS) along the projection direction. Several methods exist for computing this sequence, see for instance \cite{Loos}. However, in \cite{Bajaj}, the subresultant PRS scheme is preferred for its computational efficiency (it can be computed using, for example, the computer algebra system {\sf Maple}; for further details, see Section 5.1.2 in \cite{SPSV}).

We refer to $h(x_1, x_2)$ as the {\it lift function} since we can derive the points of the space curve ${\cal C}$ by applying $h$ to the points of the plane-projected curve ${\cal C}^p$. Additionally, note that $x_3=h(x_1,x_2)$ if and only if $h_1(x_1,x_2)-h_2(x_1,x_2)x_3=0$. Consequently, ${\cal C}$ can be implicitly defined by the polynomials $f^p(x_1,x_2)$ and $f_3(x_1,x_2,x_3)=h_2(x_1,x_2)x_3-h_1(x_1,x_2).$

\para

In Theorem \ref{T-space-projected-branches}, we investigate the relationship between the infinity branches of the space curve $\cal C$ and the infinity branches of the corresponding plane curve ${\cal C}^p$. The core idea is to utilize the lift function $h$ to derive the infinity branches of the space curve $\cal C$ from those of the plane curve ${\cal C}^p$. An efficient method for computing the infinity branches of a plane curve is presented in Section 3 of \cite{paper1}.

\para

\begin{theorem}\label{T-space-projected-branches}
$B^p=\{(z,r_2(z))\in {\Bbb C}^2:\,z\in {\Bbb C},\,|z|>M^p\},\,M^p\in {\Bbb R}^+,$  is an infinity branch of   ${\cal C}^p$ if and only if there exists   $r_3(z) \in{{\Bbb C}\ll z\gg}$, such that
$B=\{(z,r_2(z), r_3(z))\in {\Bbb C}^3:\,z\in {\Bbb C},\,|z|>M\},\,M\in{\Bbb R}^+,$ is an infinity branch of $\cal C$.
\end{theorem}
\begin{proof} Clearly, if $B$ is an infinity branch of $\cal C$, then $B^p$ is an infinity branch of ${\cal C}^p$.
Conversely, let $B^p=\{(z,r_2(z))\in {\Bbb C}^2:\,z\in {\Bbb C},\,|z|>M^p\}$ be an infinity branch  of  ${\cal C}^p$, and we look
for a series $r_3(z)=z\varphi_3(1/z)$, $\varphi_3(z)\in{{\Bbb C}\ll z\gg}$, such that $B=\{(z,r_2(z), r_3(z))\in {\Bbb C}^3:\,z\in {\Bbb C},\,|z|>M\}$ is an infinity branch of $\cal C$.
Note that, from the discussion above, we can obtain it as $r_3(z)=h(z,r_2(z))$ (observe that since $(z,r_2(z))\in {\Bbb C}^2$ for $|z|>M^p$, and $r_3(z)=h(z,r_2(z))$, we consider $(z,r_2(z),r_3(z))\in {\Bbb C}^3$ for $|z|>M$, where $M=M^p$). However, we need to prove that $r_3(z)=z\varphi_3(1/z)$ for some Puiseux series $\varphi_3(z)$.

Given $(a_1,a_2,a_3)\in{\cal C}^3$, it holds that $f_3(a_1,a_2,a_3)=h_1(a_1,a_2)-h_2(a_1,a_2)a_3=0$. Thus, in particular, $(z,r_2(z),r_3(z))\in B\subset {\cal C}$ verifies that $f_3(z,r_2(z),r_3(z))=0$. Hence,
$F_3(z,r_2(z),r_3(z),1)=0$, where $F_3(\ox)$ is the homogeneous polynomial of $f_3(x_1, x_2, x_3)$.

Taking into account the results in Section 3 of \cite{paper1}, we have that $r_2(z)=z\varphi_2(1/z)$, where
$\varphi_2(z)\in{{\Bbb C}\ll z\gg}$. Now, we search for $\varphi_3(z)\in{{\Bbb C}\ll z\gg}$ such that $r_3(z)=z\varphi_3(1/z)$. This series must satisfy (see statement above) that
$F_3(z,z\varphi_2(1/z),z\varphi_3(1/z),1)=0$ for $|z|>M$. We set $z=t^{-1}$, and we get that
$F_3(t^{-1},t^{-1} \varphi_2(t), t^{-1}\varphi_3(t),1)=0$
or equivalently $F_3(1,\varphi_2(t),\varphi_3(t),t)=0.$ This equality  holds for $|t|<1/M$ i.e., this equality must be satisfied in a neighborhood of the point at infinity $(1,\varphi_2(0),\varphi_3(0),0)$. At this point, we observe that
$$F_3(\ox)=x_4^{n_2}H_2(x_1,x_2,x_4)x_3-x_4^{n_1}H_1(x_1,x_2,x_4),$$
where $H_i(x_1,x_2,x_4)$ is the homogeneous polynomial of $h_i(x_1,x_2),\,i=1,2$, and $n_1,n_2\in\mathbb{N}$. Hence, we have
$$F_3(1,\varphi_2(t),\varphi_3(t),t)=t^{n_1}H_1(1,\varphi_2(t),t)-t^{n_2}H_2(1,\varphi_2(t),t)\varphi_3(t),$$
and since $F_3(1,\varphi_2(t),\varphi_3(t),t)=0$, we obtain that  $\varphi_3(t)=t^{n_1-n_2}\frac{H_1(1,\varphi_2(t),t)}{H_2(1,\varphi_2(t),t)}.$
Clearly, $\varphi_3(t)$ can be expressed as a Puiseux series since ${{\Bbb C}\ll t\gg}$ is a field. Therefore, we conclude that $B=\{(z,r_2(z), r_3(z))\in {\Bbb C}^3:\,z\in {\Bbb C},\,|z|>M\}$, where $r_3(z)=z\varphi_3(1/z)$, is an infinity branch of $\cal C$. \end{proof}

\para

\noindent
In the following, we illustrate the above theorem with an example.

\para

\begin{example}\label{E-infbranches}
Let $\cal C$ be the irreducible space curve defined over $\Bbb C$ by:
$$f_1(x_1,x_2,x_3)=-x_3^2+2 x_1 x_2+x_1 x_3- x_2+2,\qquad f_2(x_1,x_2,x_3)=x_3-x_1 x_2+x_2^2.$$
The projection along the $x_3$-axis, ${\cal C}^p$, is given by the polynomial:
$$f^p(x_1,x_2)=-x_1^2x_2^2 + 2 x_1  x_2^3 - x_2^4 + x_1^2 x_2 - x_1 x_2^2 + 2 x_1 x_2 - x_2 + 2$$ (computed as $\Resultant_{x_3}(f_1, f_2)$; see Section 2.3 in \cite{SWP}).

By employing the method described in \cite{paper1} (see Section 3), we compute the infinity branches of ${\cal C}_p$. For this purpose, we use the {\sf algcurves} package within the computer algebra system {\sf Maple}; specifically, we utilize the {\sf puiseux} command. We obtain the branches
\begin{itemize}
  \item  $B_1^p=\{(z,r_{1,2}(z))\in {\Bbb C}^2: \,z\in {\Bbb C},\,|z|>M_1^p\}$, where:
$$r_{1,2}(z)=z - 309034/(2187 z^4) - 7832/(243 z^3) - 221/(27 z^2) - 7/(3 z) - 2+\cdots,$$
associated with the point at infinity $P_1=(1:1:0)$
  \item $B_2^p=\{(z,r_{2,2}(z))\in {\Bbb C}^2: \,z\in {\Bbb C},\,|z|>M_2^p\}$, where:
$$r_{2,2}(z)=z + 7228/(2187 z^4) - 430/(243 z^3) + 32/(27 z^2) - 2/(3 z) + 1+\cdots,$$
associated with the point at infinity $P_2=(1:1:0)$.
  \item $B_3^p=\{(z,r_{3,2}(z))\in {\Bbb C}^2: \,z\in {\Bbb C},\,|z|>M_3^p\}$, where:
$$r_{3,2}(z)=-6/z^4 + 4/z^3 - 2/z^2+\cdots,$$
associated with the point at infinity $P_3=(1:0:0)$.
  \item $B_4^p=\{(z,r_{4,2}(z))\in {\Bbb C}^2: \,z\in {\Bbb C},\,|z|>M_4^p\}$, where:
$$r_{4,2}(z)=144/z^4 + 30/z^3 + 9/z^2 + 3/z + 1+\cdots,$$
associated with the point at infinity $P_4=(1:0:0)$.
\end{itemize}

Upon obtaining the infinity branches of the projected curve ${\cal C}^p$, we compute the infinity branches of the space curve ${\cal C}$. For this purpose, we need to compute the lift function $h(x_1, x_2)$ (applying Sections 2 and 3 in \cite{Bajaj}) to obtain the third component of these branches. In this example, we only have to compute the remainder of $f_1$ divided by $f_2$ with respect to the variable $x_3$ (see e.g., Section 5.1.2 in \cite{SPSV}). We find that $\mbox{rem}_{x_3}(f_1, f_2)=-x_1^2 x_2^2 + 2 x_1 x_2^3 - x_2^4 + x_1^2 x_2 - x_1 x_2^2 + 2 x_1 x_2 - x_2 + 2$. Thus, the lift function $h(x_1, x_2)$ is obtained by solving the equation $f_2=0$ for the variable $x_3$. We get that $h(x_1,x_2)=x_1 x_2-x_2^2$, and thus, the infinity branches of the space curve are

\begin{itemize}
  \item  $B_1=\{(z,r_{1,2}(z),r_{1,3}(z))\in {\Bbb C}^3: \,z\in {\Bbb C},\,|z|>M_1\}$, where:
$$r_{1,3}(z)=-56456/(2187 z^3) - 1447/(243 z^2) - 31/(27 z) - 5/3 + 2 z  - 1711603/(2187 z^4)+\cdots,$$
associated with the point at infinity $P_1=(1:1:2:0)$
  \item $B_2=\{(z,r_{2,2}(z),r_{2,3}(z))\in {\Bbb C}^3: \,z\in {\Bbb C},\,|z|>M_2\}$, where:
$$r_{2,3}(z)=3968/(2187 z^3) - 254/(243 z^2) + 4/(27 z) - 1/3 - z   - 22688/(2187 z^4)+\cdots,$$
associated with the point at infinity $P_2=(1:1:-1:0)$.
  \item $B_3=\{(z,r_{3,2}(z),r_{3,3}(z))\in {\Bbb C}^3: \,z\in {\Bbb C},\,|z|>M_3\}$, where:
$$r_{3,3}(z)=-6/z^3 + 4/z^2 - 2/z   - 4/z^4+\cdots,$$
associated with the point at infinity $P_3=(1:0:0:0)$.
  \item $B_4=\{(z,r_{4,2}(z),r_{4,3}(z))\in {\Bbb C}^3: \,z\in {\Bbb C},\,|z|>M_4\}$, where:
$$r_{4,3}(z)=30/z^3 + 3/z^2 + 3/z + 2 + z  - 549/z^4+\cdots,$$
associated with the point at infinity $P_4=(1:0:1:0)$.
\end{itemize}
In Figure \ref{F-ejemplo-branches}, we plot the curve $\cal C$ and some points of the branches $B_i,\,i=1,\ldots,4$.

\para

\begin{center}
   \begin{figure}[h]
\hspace{3.5cm}
  \includegraphics[width=0.60\textwidth]{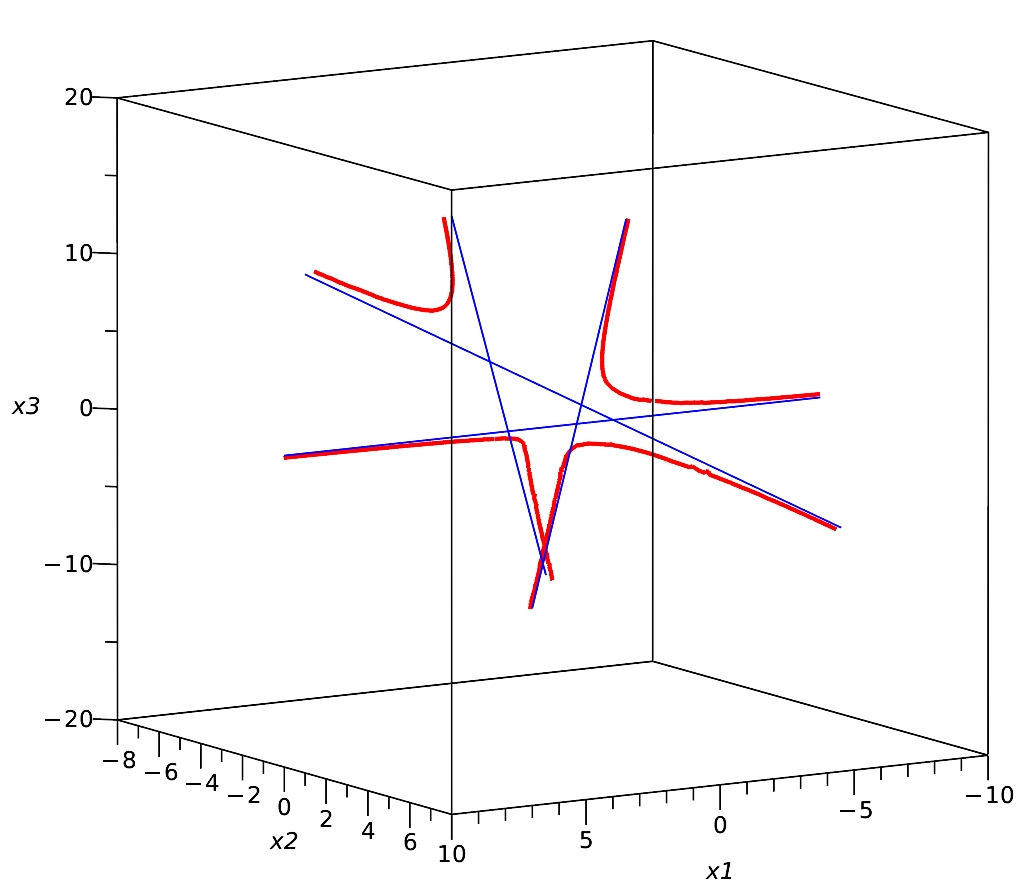}
        \caption{Curve $\cal C$ (red color) and infinity branches (lines in blue color).}\label{F-ejemplo-branches}
 \end{figure}
\end{center}

 \end{example}

\section{Asymptotes of a given infinity branch}\label{S-computationasymptote}

In \cite{BlascoPerezII} (see Section 3), we demonstrate how certain algebraic plane curves can be approached at infinity by curves of lower degree. A well-known example is the case of hyperbolas, which are degree $2$ curves approached at infinity by two lines (their asymptotes). Similar situations may arise when dealing with curves of higher degree.

\para

Determining the asymptotes of an implicitly defined algebraic curve is an important topic covered in many analysis textbooks (see e.g., \cite{maxwell}). Some algorithms for computing the linear asymptotes of a real plane algebraic curve can be found in the literature (see e.g., \cite{jovan}, \cite{RSS}, or \cite{Zeng}). However, as shown in \cite{BlascoPerezII}, an algebraic plane curve may have more general curves than lines describing the behavior of a branch at points with sufficiently large coordinates. The theory and practical methods concerning these special curves, called {\it generalized asymptotes}, are presented in \cite{BlascoPerezII} (see Sections 3, 4, and 5) for the case of plane curves.

\para

In this section, we aim to study and compute the {\it generalized asymptotes} for a given algebraic space curve. We use the concepts and results outlined in \cite{paper1} to craft an efficient algorithm for computing the g-asymptotes of $\mathcal{C}$. The strategy involves obtaining the asymptotes of $\mathcal{C}$ by first determining those of ${\mathcal{C}^p}$, where ${\mathcal{C}^p}$ represents a planar curve birationally equivalent to the given spatial curve. Thus, the task boils down to establishing the equivalence between the asymptotes of these curves and devising an effective algorithm for computing the asymptotes of $\mathcal{C}$ once those of ${\mathcal{C}^p}$ are identified. This concept echoes the approach used in Theorem \ref{T-space-projected-branches}, where the birational correspondence between spatial and planar branches is demonstrated. Furthermore, leveraging the effective results presented in the author's previous works, one can compute the asymptotes of spatial curves very efficiently, avoiding the use of Puiseux series. It's evident that this methodology readily extends to curves in $n$-dimensional space.

\para

Some results cannot be seen as a straightforward generalization from the case of plane curves. Although the construction of asymptotes is similar (in the sense that a new component needs to be computed for the space curve), the formalization of the results as well as the detailed proofs require a different perspective (for instance, the computation of the degree of a space curve is entirely different from the computation of the degree of a plane curve). Thus, we need new tools and different reasoning compared to the plane case.

\para

We start  with some important definitions and previous results that will lead to the construction of  asymptotes in Subsection  \ref{Sub-construction-asymptote}.

\para

\begin{definition}\label{D-perfect-curve}
A curve ${\cal C}$ of degree $d_{\cal C}$ is a {\em perfect curve} if it cannot be approached by any curve of degree less than $d_{\cal C}$.
\end{definition}

\para

A curve that is not perfect can be approached by other curves of less degree. If these curves are perfect, we call them {\it g-asymptotes}. More precisely, we have the following definition.

\para

\begin{definition}\label{D-asymptote}
Let ${\cal C}$ be a curve with an infinity branch $B$. A {\em g-asymptote} (generalized asymptote) of ${\cal C}$ at $B$ is a perfect curve that approaches ${\cal C}$ at $B$.
\end{definition}

\para

The notion of {\em g-asymptote} is similar to the classical concept of asymptote. The difference is that a g-asymptote does not have to be a line, but a perfect curve. Actually, it is a generalization, since every line is a perfect curve (this remark follows from Definition \ref{D-perfect-curve}). Throughout the paper, we refer to {\em g-asymptote} simply as {\it asymptote}.

\para

\begin{remark}\label{R-minimal-degree}
The degree of an asymptote is less than or equal to the degree of the curve it approaches. In fact, an asymptote of a curve $\cal C$ at a branch $B$ has minimal degree among all the curves that approach $\cal C$ at $B$ (see Remark 3 in \cite{BlascoPerezII}).
\end{remark}

\para

In the following, we prove that every infinity branch of a given algebraic space curve has, at least, one asymptote and we show how to obtain it. Most of the results introduced below for the case of space curves generalize the results presented in \cite{BlascoPerezII} for the plane case.

\para

Let $\cal C$ be an irreducible space curve implicitly defined by the polynomials $f_i\in {\Bbb R}[x_1,x_2,x_3],$ $i=1,2$, and let $B=\{(z,r_2(z), r_3(z))\in {\Bbb C}^3:\,z\in {\Bbb C},\,|z|>M\}$ be an infinity branch of $\cal C$ associated with the point at infinity $P=(1:m_2:m_3:0)$. We know that $r_2(z)$ and $r_3(z)$ are given as follows:
$$r_{2}(z)=m_2z+a_{1,2}z^{1-N_{1,2}/N_2}+a_{2,2}z^{1-N_{2,2}/N_2}+a_{3,2}z^{1-N_{3,2}/N_2} + \cdots$$
$$r_{3}(z)=m_3z+a_{1,3}z^{1-N_{1,3}/N_3}+a_{2,3}z^{1-N_{2,3}/N_3}+a_{3,3}z^{1-N_{3,3}/N_3} + \cdots$$
where $a_{i,2}\neq0$, $N_2, N_{i,2}\in {\Bbb N},\,i\geq 1$, $0<N_{1,2}<N_{2,2}<\cdots$, and $a_{i,3}\neq0$, $N_3, N_{i,3}\in {\Bbb N},\,i\geq 1$, and $0<N_{1,3}<N_{2,3}<\cdots$. Let $N:=\operatorname{lcm}(N_2,N_3)$, and note that $\nu(B)=N$. It holds that $\deg({\cal C})\geq N$ (see \cite{NewPara-As}).

\para

In the following, let $\ell_k,\,k=2,3$, be the first integer verifying that $N_{\ell_k, k}\leq N_k<N_{\ell_k+1, k}$. Then, we can write
$$
\begin{array}{l}
r_{2}(z)=m_2z+a_{1,2}z^{1-\frac{N_{1,2}}{N_2}}+\cdots +a_{\ell_2, 2}z^{1-\frac{N_{\ell_2, 2}}{N_2}}+a_{\ell_2+1, 2}z^{1-\frac{N_{\ell_2+1, 2}}{N_2}}+\cdots\\
r_{3}(z)=m_3z+a_{1,3}z^{1-\frac{N_{1,3}}{N_3}}+\cdots +a_{\ell_3, 3}z^{1-\frac{N_{\ell_3, 3}}{N_3}}+a_{\ell_3+1, 3}z^{1-\frac{N_{\ell_3+1, 3}}{N_3}}+\cdots
\end{array}
$$
where the exponents $1-\frac{N_{j,k}}{N_k}$ are non-negative for $j\leq\ell_{k}$ and negative for $j>\ell_{k}$.   Now, we simplify (if necessary) the non-negative exponents and rewrite the above expressions, and we get

\begin{equation}\label{Eq-inf-branchn}
\begin{array}{l}
r_{2}(z)=m_2z+a_{1,2}z^{1-\frac{n_{1,2}}{n_2}}+\cdots +a_{\ell_2, 2}z^{1-\frac{n_{\ell_2, 2}}{n_2}}+a_{\ell_2+1, 2}z^{1-\frac{N_{\ell_2+1, 2}}{N_2}}+\cdots\\
r_{3}(z)=m_3z+a_{1,3}z^{1-\frac{n_{1,3}}{n_3}}+\cdots +a_{\ell_3, 3}z^{1-\frac{n_{\ell_3, 3}}{n_3}}+a_{\ell_3+1, 3}z^{1-\frac{N_{\ell_3+1, 3}}{N_3}}+\cdots
\end{array}
\end{equation}
where $\operatorname{gcd}(n_k,n_{1,k},\ldots,n_{\ell_k,k})=1$, $k=1,2$. Note that $0<n_{1,k}<n_{2,k}<\cdots$, $n_{\ell_k,k}\leq n_k$.

\para

Under these conditions, we introduce the definition of degree of a branch $B$ as follows:

\para

\begin{definition}\label{D-degreebranch} Let $B=\{(z,r_2(z), r_3(z))\in {\Bbb C}^3: \,z\in {\Bbb C},\,|z|>M\}$ defined by  (\ref{Eq-inf-branchn}), be an infinity branch associated with $P=(1:m_2:m_3:0)$, $m_j\in {\Bbb C}$, $j=1,2$. We say that $n:=\operatorname{lcm}(n_2, n_3)$ is the degree of $B$, and we denote it by $\deg(B)$.
\end{definition}

\para
Note that $n_i\leq N_i$, $i=1,2$. Thus, $n=\operatorname{lcm}(n_2, n_3)=\deg(B)\leq N=\operatorname{lcm}(N_2, N_3)$, and since  $\deg({\cal C})\geq N$, we get that $\deg({\cal C})\geq \deg(B)$. In fact, if $\overline{{\cal C}}$ is a curve that approaches ${\cal C}$ at its infinity branch $B$, then  that $\deg(\overline{{\cal C}})\geq \deg(B)$.

\subsection{Construction of
asymptotes}\label{Sub-construction-asymptote}

In this subsection, we present an algorithm that allows computing an asymptote for each of the infinity branches of a given implicit space curve. An example illustrating the algorithm is also presented.

\para

The algorithm is derived from the results presented above and the construction developed throughout this subsection. We formally show how to construct  an asymptote of the given space curve that can be easily parametrized, and we prove that this parametrization is proper. Although the results are equivalent to those presented for the plane case, the proofs and detailed discussions have to be different since the tools used to deal with the space curve differ from those used in the plane case  (see Section 3 in \cite{BlascoPerezII}).

\para

Let ${\cal C}$ be a space curve with an infinity branch $B=\{(z,r_2(z),r_3(z))\in {\Bbb C}^3:\,z\in {\Bbb C},\,|z|>M\}$. Taking into account the results presented above, we have that any curve $\overline{{\cal C}}$ approaching ${\cal C}$ at $B$ has an infinity branch $\overline{B}=\{(z,\overline{r}_2(z),\overline{r}_3(z))\in {\Bbb C}^3:\,z\in {\Bbb C},\,|z|>\overline{M}\}$ such that the terms with non-negative exponents in $r_i(z)$ and $\overline{r}_i(z)$ (for $i=2,3$) are the same. We consider the series $\tilde{r}_{2}(z)$ and $\tilde{r}_{3}(z)$, obtained from $r_{2}(z)$ and $r_{3}(z)$ by removing the terms with negative exponents (see equation (\ref{Eq-inf-branchn})). Then, we have that

\begin{equation}\label{Eq-inf-branch3}
\begin{array}{l}
\tilde{r}_{2}(z)=m_2z+a_{1,2}z^{1-n_{1,2}/n_2}+\cdots +a_{\ell_2, 2}z^{1-n_{\ell_2, 2}/n_2}\\
\tilde{r}_{3}(z)=m_3z+a_{1,3}z^{1-n_{1,3}/n_3}+\cdots +a_{\ell_3, 3}z^{1-n_{\ell_3, 3}/n_3}
\end{array}
\end{equation}
where $a_{j,k},\ldots\in\mathbb{C}\setminus \{0\}$, $m_k\in {\Bbb C}$, $n_k, n_{j,k},\ldots\in\mathbb{N}$, $\gcd(n_k,n_{1,k},\ldots,n_{\ell,k})=1$, and $0<n_{1,k}<n_{2,k}<\cdots$. That is, $\tilde{r}_k$ has the same terms with non-negative exponents as $r_k$, and $\tilde{r}_k$ does not have terms with negative exponents.

\para

Let $\widetilde{{\cal C}}$ be the space curve containing the branch $\widetilde{B}=\{(z,\tilde{r}_2(z),\tilde{r}_3(z))\in {\Bbb C}^3:\,z\in {\Bbb C},\,|z|>\widetilde{M}\}$. Observe that
\[\widetilde{{\cal Q}}(t)=(t^n, m_2t^n+a_{1,2}t^{\upsilon_2(n_2-n_{1,2})}+\cdots +a_{\ell_2,2}t^{\upsilon_2(n_2-n_{\ell_2,2})},\]
\begin{equation}\label{Eq-parametric-case1}
m_3t^n+a_{1,3}t^{\upsilon_3(n_3-n_{1,3})} +\cdots +a_{\ell_3,3}t^{\upsilon_3(n_3-n_{\ell_3,3})})\in {\Bbb C}[t]^3,
\end{equation}
where $n=\lcm(n_2, n_3)$, $\upsilon_k=n/n_k$, $n_k, n_{1,k},\ldots,n_{\ell_k,k}\in\mathbb{N}$, $0<n_{1,k}<\cdots<n_{\ell_k, k}$, and $\gcd(n_k,n_{1,k},\ldots,$ $n_{\ell_k, k})=1$, is a polynomial parametrization of $\widetilde{{\cal C}}$. In addition,   we prove that $\widetilde{{\cal Q}}$ is proper (i.e., invertible), and $\deg(\widetilde{{\cal C}})=\deg(B)$. Hence, the   curve $\widetilde{{\cal C}}$ is an asymptote of $\mathcal{C}$ at $B$.

\para

From this reasoning, in the following, we present an algorithm that computes an asymptote for each infinity branch of a given space curve. We assume that we have prepared the input curve $\mathcal{C}$, by means of a suitable linear change of coordinates if necessary, such that $(0:a:b:0)$ ($a\neq0$ or $b\neq 0$) is not a point at infinity of $\mathcal{C}$. We consider the birational correspondence between the points of $\mathcal{C}^p$ and the points of $\mathcal{C}$, where $\mathcal{C}^p$ is the plane curve obtained by projecting $\mathcal{C}$ along the $x_3$-axis (see Subsection \ref{Sub-computationinfinitybranches}).

\para

We should note that the significant issue lies in the inefficiency of calculation techniques, as they necessitate computing the entire infinity branch repeatedly using Puiseux series.

\para

\begin{center}
\fbox{\hspace*{2 mm}\parbox{15 cm}{ \vspace*{2 mm} {\bf Algorithm
{\sf Space Asymptotes Construction.}} \vspace*{0.2cm}

\noindent {\sf Given} an irreducible real algebraic space curve $\cal C$
implicitly defined
 by two polynomials $f_1(x_1,x_2,x_3),f_2(x_1,x_2,x_3)\in {\Bbb R}[x_1,x_2,x_3]$,  the algorithm {\sf outputs} an
asymptote for each of its infinity branches.

\begin{itemize}

\item[1.] Compute the projection of $\cal C$ along
the $x_3$-axis. Let  ${\cal C}_p$ be this projection and $f^p(x_1,x_2)$ the implicit polynomial defining ${\cal C}_p$.

\item[2.] Determine the lift function $h(x_1,x_2)$ (see
Sections 2 and 3 in \cite{Bajaj}).

\item[3.] Compute the  infinity branches of ${\cal C}_p$
by applying  the results in Section 3 in \cite{paper1}.

\item[4.] For each branch $B^p_i=\{(z,r_{i,2}(z))\in {\Bbb
C}^2:\,z\in {\Bbb C},\,|z|>M^p_{i,2}\}$, $i=1,\ldots, s,$ do:
\begin{itemize}
\item[4.1.] Compute the corresponding infinity branch of $\cal C$:
$$B_i=\{(z,r_{i,2}(z),r_{i,3}(z))\in {\Bbb
C}^3:\,z\in {\Bbb C},\,|z|>M_i\}$$ where
$r_{i,3}(z)=h(z,r_{i,2}(z))$ is given as a Puiseux series.

\item[4.2.] Consider the series $\tilde{r}_{i,2}(z)$ and
$\tilde{r}_{i,3}(z)$  obtained by eliminating the terms with
negative exponent in $r_{i,2}(z)$ and $r_{i,3}(z)$, respectively  (see equation (\ref{Eq-inf-branch3})).

\item[4.3.] Return  the asymptote $\widetilde{\cal C}_i$ defined by the proper parametrization,
$\widetilde{Q}_i(t)=(t^{n_i},\,\tilde{r}_{i,2}(t^{n_i}),\,\tilde{r}_{i,3}(t^{n_i}))\in
{\Bbb C}[t]^3$, where $n_i=\deg(B_i)$ (see Definition
\ref{D-degreebranch}).
\end{itemize}\end{itemize}
}\hspace{2 mm}}
\end{center}

\para

\begin{remark}\label{R-algoritmo}\begin{itemize}
\item[1.] The implicit polynomial $f^p(x_1,x_2)$  defining ${\cal C}_p$ (see step 1) can be computed as $f^p(x_1,x_2)=\resultant_{x_3}(f_1,f_2)$ (see Section 4.5 in \cite{SWP}).
\item[2.] Since   we have assumed that the given algebraic space curve $\cal C$ only has points at infinity of the form $(1 : m_2 : m_3 : 0)$, we have that $(0:1:0)$ is not a point  at infinity of the plane curve ${\cal C}_p$. Thus, results in Section 3 in \cite{paper1} can be applied.  
       
    \end{itemize}
\end{remark}

\para

\noindent
In the following example, we  illustrate algorithm {\sf Space Asymptotes Construction.} \para

\para

\begin{example}
Let $\mathcal{C}$ be the algebraic space curve over $\mathbb{C}$ introduced in Example \ref{E-infbranches}. In Example \ref{E-infbranches}, we show that $\mathcal{C}$ has four infinity branches $B_i=\{(z,r_{i,2}(z),r_{i3,}(z))\in \mathbb{C}^3: \,z\in \mathbb{C},\,|z|>M_i\},\, i=1,\ldots,4.$

These branches were obtained by applying steps 1, 2, 3, and 4.1 of Algorithm \textsf{Space Asymptotes Construction}. Now we apply step 4.2, and we compute the series $\tilde{r}_{i,j}(z)$ by removing the terms with negative exponent from the series $r_{i,j}(z)$, $i=1,\ldots,4$, $j=2,3$. We get:
$$
\begin{array}{llll}
\displaystyle
\tilde{r}_{1,2}(z)=z - 2,   & \displaystyle\tilde{r}_{2,2}(z)=z + 1, & \displaystyle\tilde{r}_{3,2}(z)=0, & \displaystyle\tilde{r}_{4,2}(z)=1,\\
\\
\displaystyle \tilde{r}_{1,3}(z)=-5/3 + 2z,   &
\displaystyle\tilde{r}_{2,3}(z)= -1/3 - z, &
\displaystyle\tilde{r}_{3,3}(z)=0, &
\displaystyle\tilde{r}_{4,3}(z)=2 + z.
\end{array}
$$
Thus, in step 4.3, we obtain:
$$\widetilde{Q}_1(t)=(t,\,\tilde{r}_{1,2}(t),\,\tilde{r}_{1,3}(t))=(t, t - 2, -5/3 + 2 t),\quad  \widetilde{Q}_2(t)=(t,\,\tilde{r}_{2,2}(t),\,\tilde{r}_{2,3}(t))=(t,t + 1, -1/3 - t),$$
  $$\widetilde{Q}_3(t)=(t,\,\tilde{r}_{3,2}(t),\,\tilde{r}_{3,3}(t))=(t,0,0),\quad \widetilde{Q}_4(t)=(t,\,\tilde{r}_{4,2}(t),\,\tilde{r}_{4,3}(t))=(t,1,2+t).$$
 $\widetilde{Q}_i$ are proper parametrizations  of the asymptotes $\widetilde{{\cal C}}_i$, which approach $\mathcal{C}$ at its infinity branches $B_i$, for $i=1,\ldots,4$. In Figure \ref{F-ejemplo-asintotas}, we plot the curve $\mathcal{C}$ and its asymptotes $\widetilde{{\cal C}}_i$.

\begin{center}
   \begin{figure}[h]
   $\begin{array}{ccc}
    \includegraphics[width=0.31\textwidth]{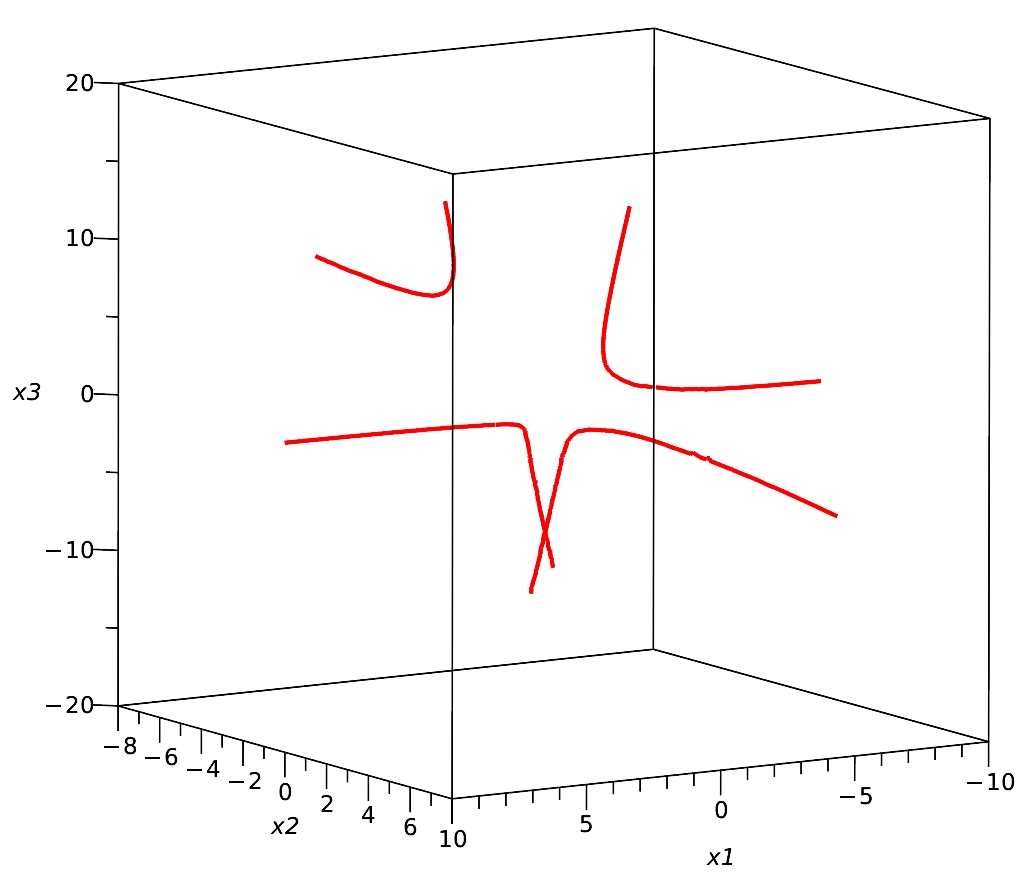}  & \includegraphics[width=0.315\textwidth]{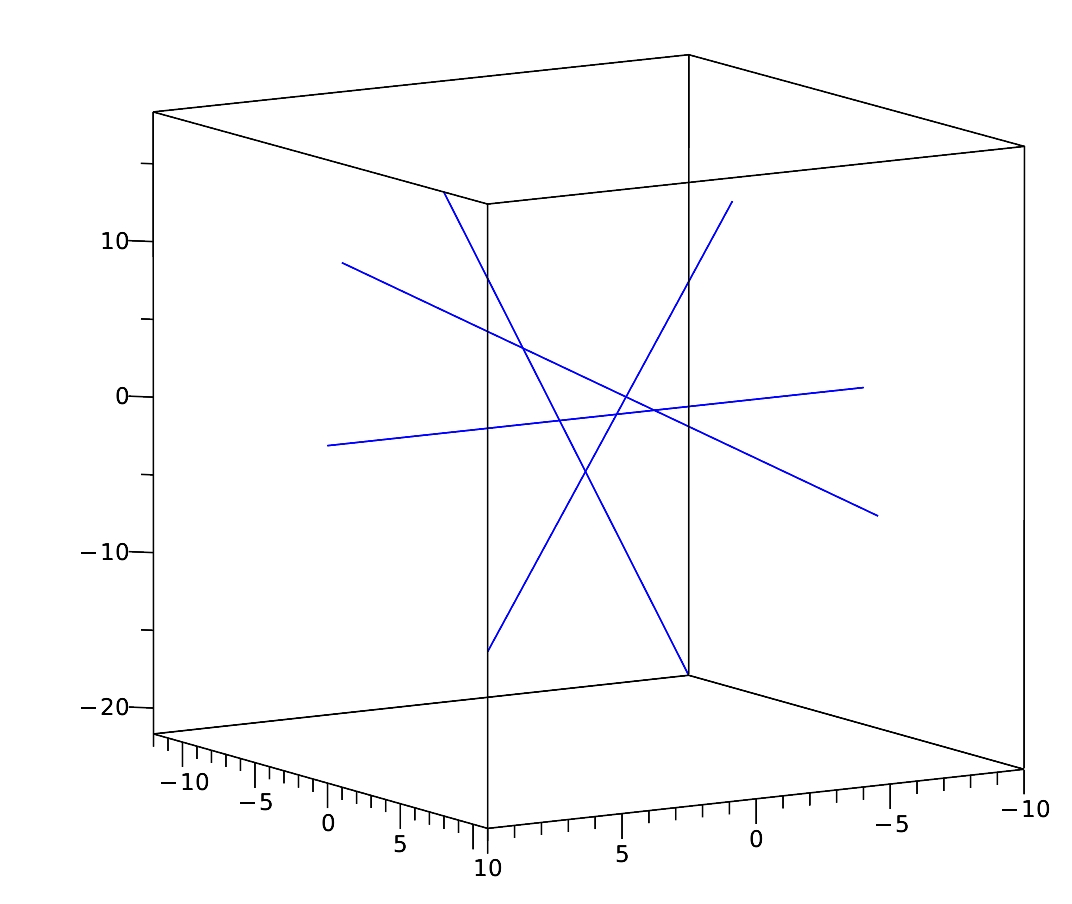}&  \includegraphics[width=0.31\textwidth]{fig13}
   \end{array}$
        \caption{Curve $\cal C$ (left),  asymptotes (center) and curve and asymptotes (right).}\label{F-ejemplo-asintotas}
 \end{figure}
\end{center}

\end{example}

\subsection{New efficient method for the computation of asymptotes}\label{Sub-improvement}

In this section, we introduce an enhancement of the method described above, which circumvents the need for computing infinity branches and Puiseux series. Instead, we only need to identify the solutions of a triangular system of equations derived from the implicit polynomial.

\para

It's worth mentioning that we assume that all the infinity points are of the form $(1:m_2:m_3:0)$; otherwise, we apply a change of coordinates.

\para

We provide the main theorem and a corollary, outlining a constructive approach for determining all the associated asymptotes (also the corresponding branch) to the infinity points. For this purpose, we denote by $F_3(x_1,x_2,x_3,x_4)=x_3H_2(x_1,x_2,x_4)-H_1(x_1,x_2,x_4)$ the homogenization of $f_3(x_1,x_2,x_3)=x_3h_2(x_1,x_2)-h_1(x_1,x_2)$ and  $B_p=\{(z,r_2(z))\in {\Bbb C}^2:\,z\in {\Bbb C},\,|z|>M\},$
$$r_2(z)=a_0z+a_1z^{(k-1)/k}+a_2z^{(k-2)/k}+\cdots+a_k+a_{k+1}z^{-1/k}+a_{k+2}z^{-2/k}+ \cdots,$$
is an infinity branch of the plane curve ${\cal C}_p$  defined by   $f^p(x_1,x_2) \in \mathbb{R}[x_1,x_2]$. Observe that the asymptote of ${\cal C}_p$ is given by $(t^k, q_2(t))$, where $q_2(t)=a_0t^k+a_1t^{k-1}+a_2t^{k-2}+\cdots+a_k$.
  \para

  Finally, in the following $b_j,\,j=0,\ldots, k$, denote some undetermined coefficients. We start with the following technical lemma.

\para

\begin{lemma}\label{L-final} There exists $n_0\in {\Bbb N}$ such that for every $n\geq n_0,\,\,n\in {\Bbb N}$, it holds that
\[F_3(t^{n+k},a_0t^{n+k}+a_1t^{n+k-1}+\cdots+a_{n+k}, t^n(b_0 t^{k}+b_{1}t^{k-1}+\cdots+b_{k}),  t^n)=\]
$$\Lambda_0(b_0) t^{\mu}+\Lambda_{1}(b_0, b_{1}) t^{\mu-1}+\cdots+\Lambda_k(b_0, b_{1},\ldots,b_k) t^{\mu-k}+\sum_{j=k+1}^{\mu} \Lambda_j(b_0, b_{1},\ldots,b_k) t^{\mu-j},$$
where
$$\Lambda_0(b_0)=b_0 \rho(a_0,\ldots,a_r)-\rho_0(a_0,\ldots,a_r),\qquad \rho(a_0,\ldots,a_r)\not=0,$$
$$\Lambda_j(b_0, b_{1},\ldots,b_j)=b_j \rho(a_0,\ldots,a_r)-\rho_j(a_0,\ldots,a_{r+j}),\qquad j=0,\ldots, k$$
 $\mu-k\geq 0$ and $r\leq n$.
\end{lemma}
\begin{proof} First, let $n\in {\Bbb N}$ be such that $$g_t(s):=H_2(t^{n+k},a_0t^{n+k}+a_1t^{n+k-1}s+\cdots+a_{n+k}s^{n+k}, t^n(b_0 t^{k}+b_{1}t^{k-1}s+\cdots+b_{k}s^k),  t^ns^k)\not=0,$$
and let $j_0\in {\Bbb N}$ be the first natural number  such that $g_t^{j_0)}(0)\not=0$ ($g_t^{j)}$ the partial derivative or order $j$ w.r.t $s$). Note that clearly $j_0\in {\Bbb N}$ exists and $\deg(g_t^{j}(0))=(k+n)\deg(H_2)-j$ (if $g_t^{j)}(0)\not=0$). Now, from
\[F_3(t^{n+k},a_0t^{n+k}+a_1t^{n+k-1}+\cdots+a_{n+k}, t^n(b_0 t^{k}+b_{1}t^{k-1}+\cdots+b_{k}),  t^n)=\]

\noindent
$t^n(b_0 t^{k}+b_{1}t^{k-1}+\cdots+b_{k})H_2(t^{n+k},a_0t^{n+k}+a_1t^{n+k-1}+\cdots+a_{n+k}, t^n)-H_1(t^{n+k},a_0t^{n+k}+a_1t^{n+k-1}+\cdots+a_{n+k}, t^n)= t^n(b_0 t^{k}+b_{1}t^{k-1}+\cdots+b_{k})\left(g_t^{j_0)}(0)/j_0!+g_t^{j_0+1)}(0)/2!+\cdots+g_t^{(k+n)\deg(H_2))}(0)/((k+n)\deg(H_2)!)\right)-H_1(t^{n+k},a_0t^{n+k}+a_1t^{n+k-1}+\cdots+a_{n+k}, t^n)$
$$=\Lambda_0(b_0) t^{\mu}+\Lambda_{1}(b_0, b_{1}) t^{\mu-1}+\cdots+\Lambda_k(b_0, b_{1},\ldots,b_k) t^{\mu-k}+\sum_{j=k+1}^{\mu} \Lambda_j(b_0, b_{1},\ldots,b_k) t^{\mu-j},$$
we get that  $\mu=(k+n)\deg(H_2)-j_0$, and
$$\Lambda_0(b_0)=b_0 \rho(a_0,\ldots,a_r)-\rho_0(a_0,\ldots,a_r), \qquad \rho(a_0,\ldots,a_r)=g_{t=1}^{j_0)}(0)\not=0$$
 $$\Lambda_j(b_0, b_{1},\ldots,b_j)=b_j \rho(a_0,\ldots,a_r)-\rho_j(a_0,\ldots,a_{r+j},b_0, b_{1},\ldots,b_{j-1}),\qquad j=0,\ldots, k$$
 with $\Lambda_j\not=0,\,j=0,\ldots, k$.  Observe that since  we may write $H_1$ similarly as $H_2$, we get that in  each $\Lambda_j\not=0$ the $a_j$'s from $r$ are appearing one by one until all $a_j$'s  have appeared, at which point all the coefficients $\Lambda_j$ depend on all  $a_0,\ldots, a_{n+k}$.\\

Therefore, we  finally consider $n\in {\Bbb N}$ such that $\mu=(k+n)\deg(H_2)-j_0\geq k$, and $n+k\geq r+k$ and thus $n_0=\max\{r,\,(k+j_0)/\deg(H_2)-k\}$.
\end{proof}

\para

\begin{remark}\label{R-final}
\begin{enumerate}
\item Note that
\[F_3(t^{n+k},a_0t^{n+k}+a_1t^{n+k-1}+\cdots+a_{n+k}, t^n(b_0 t^{k}+b_{1}t^{k-1}+\cdots+b_{k}),  t^n)\]
is the numerator of the rational function
\[f_3(t^{k},a_0t^{k}+a_1t^{k-1}+\cdots a_k+a_{k+1}/t+\cdots+a_{n+k}/t^n, b_0 t^{k}+b_{1}t^{k-1}+\cdots+b_{k}).\]

\item In order to determine  $r\in {\Bbb N}$ in  Theorem \ref{T-final}, we observe that $r$ is defined by the first $r$ terms $\{a_0, a_1, \ldots, $ $a_{n+k}\}$ that appear in  $\Lambda_0$ (note that if $a_{j}=0,\,j=0,\ldots,\ell$ then $r\geq \ell$). It is easy to determine which terms these $a_j$'s, $j=1, \ldots, r$, are by simply noting that if we modify the term $a_{r+1}$, the coefficient $\Lambda_0$ does not change. For this purpose, one may proceed as follows: we consider $v_1,\ldots, v_k$ undetermined parameters and
\[g(t):=f_3(t^{k},a_0t^{k}+v_1t^{k-1}+\cdots+v_{k}, b_0 t^{k}+b_{1}t^{k-1}+\cdots+b_{k}).\]
If the leader coefficient of $g$ depends on $v_1,\ldots,v_{\ell},\,\,\,\ell< k$, then $r=\ell$. If $\ell=k$, then we consider
$g(t)$ the numerator of \[f_3(t^{k},a_0t^{k}+a_1t^{k-1}+\cdots+a_{k}+v_{k+1}/t, b_0 t^{k}+b_{1}t^{k-1}+\cdots+b_{k})\] and reason as above. Once we have $r$, we consider $n=r$ and we check weather  $\mu\geq k$. If it does not hold, we consider $n+1$ and check again the condition.\para

In any case, in the practice,  it suffices to consider a sufficiently large  $n\in {\Bbb N}$ to obtain the correct result. Observe that calculating the branch $r_2(z)$ with more terms is not a problem at all, nor is it computationally expensive, because one only needs to solve easy triangular systems by applying the results from \cite{NewImpl-As}.\\
\end{enumerate}
 \end{remark}

\para

Under the conditions outlined in Lemma \ref{L-final} and using the notation introduced therein, we derive the following theorem.

\para

\begin{theorem}\label{T-final} Let $(t^k, q_2(t))$ be an asymptote of   ${\cal C}_p$  obtained from the branch $B_p$.
The g-asymptotes of $\mathcal{C}$ are defined by the parametrizations
$\widetilde{\mathcal{Q}}(t) = (t^{k},\,  q_2(t), q_3(t))$
with $q_3(t)=b_0 t^{k}+b_{1}t^{k-1}+\cdots+b_{k},\,$ $b_i \in \mathbb{C},\,i=0,\ldots,k$ satisfying that
$$\Lambda_i(b_0,\ldots, b_i)=0,\,\,i=0,\ldots,k.$$
\end{theorem}
\begin{proof} The proof strategy revolves around the observation that for a given branch, the substitution in the implicit function $f_3(x_1,x_2,x_3)$ must converge at infinity, implying that terms with positive exponents in this substitution must be $0$ (see equation \ref{E-uno}). \\
To begin, let $B=\{(z,r_2(z),r_3(z))\in \mathbb{C}^2: \,z\in \mathbb{C},\,|z|>M\}$, $M\in \mathbb{R}^+$, and
$$ r_3(z)=z\varphi_3(z^{-1})=b_0z+b_1z^{1-N_1/N}+b_2z^{1-N_2/N}+b_3z^{1-N_3/N}
+ \cdots$$ where $N, N_i\in \mathbb{N},\,\,i\in \mathbb{N}$, $0<N_1<\cdots<N_k= N<N_{k+1}<\cdots$, and
 \[\varphi_3(z)=b_0+b_1z^{N_1/N}+b_2z^{N_2/N}+b_3z^{N_3/N}+ \cdots.\]
 We reason similarly for $r_2(z)$, and we get that
  \[\varphi_2(z)=a_0+a_1z^{N_1/N}+a_2z^{N_2/N}+a_3z^{N_3/N}+ \cdots.\]
Now, let
   \[\gamma_3(z)=\varphi_3(z^N)=b_0+b_1z^{N_1}+b_2z^{N_2}+b_3z^{N_3}+ \cdots\]
   and   $\Gamma_3(z,w)$ be the homogenization of $\gamma_3(z)$ (similarly for $\gamma_2(z)$ and  $\Gamma_2(z,w)$). We have that\\[-0.2cm]

       \noindent $F_3(z^N, z^N\gamma_2(z^{-1}), z^N\gamma_3(z^{-1}),1)=$ \\
      \noindent$F_3(z^N,\, a_0z^N+a_1z^{N-N_1}+ \cdots+a_k+a_{k+1}z^{N-N_{k+1}}+ \cdots,  b_0z^N+b_1z^{N-N_1}+ \cdots+b_k+b_{k+1}z^{N-N_{k+1}}+ \cdots,\,1)=0,$\\[0.1cm]
   \noindent
   which is equivalent to
\begin{equation}\label{E-uno}F_3(1,\Gamma_2(1,w),\Gamma_3(1,w),w^N)=\end{equation}

\noindent
$F_3(1,\,  a_0+a_1w^{N_1}+\cdots+a_k w^{N}+a_{k+1}w^{N_{k+1}}+\cdots,\, b_0+b_1w^{N_1}+\cdots+b_k w^{N}+b_{k+1}w^{N_{k+1}}+\cdots, \, w^N)=0$\\

\noindent
where $N-N_j\geq 0,\,j=1,\ldots,k$ and $N-N_j<0,\,j\geq k+1$. Let us denote $${\cal E}_1(w):=F_3(1,\Gamma_2(1,w),\Gamma_3(1,w),w^N).$$
Next, consider the truncated branch. Specifically, let  $$r_3^\star(z)=b_0z^k+b_1z^{k-1}+b_2z^{k-2}+\cdots+b_k$$ and
 $$r_2^\star(z)=a_0z^k+a_1z^{k-1}+a_2z^{k-2}+\cdots+a_k+a_{k+1}/z+\cdots+a_{k+n}/z^n$$
 where  $n\in {\Bbb N}$ is given by the maximum of $n_0$ (obtained in Lemma \ref{L-final}) and $k-\ell$ ($\ell$ is given such that $a_{j}=0,\,\,j=0,\ldots,\ell$).

\para

 Then, we consider the polynomial
$F_3(z^{k+n},R_2^\star(z,w),z^nR_3^\star(z,w), w^k z^n)$,
where $R_i^\star(z,w),\,i=2,3$ are the homogenizations of $r_i^\star(z),\,i=2,3,$, respectively; i.e.
$$R_2^\star(z,w)=a_0z^{k+n}+a_1z^{k+n-1}w+a_2z^{k+n-2}w^2+\cdots+a_kz^nw^k+a_{k+1}z^{n-1}w^{k+1}+\cdots+a_{k+n}w^{k+n}$$
$$R_3^\star(z,w)=b_0z^{k}+b_1z^{k-1}w+b_2z^{k-2}w^{2}+\cdots+b_kw^{k}.$$
Let
\begin{equation}\label{E-dos} {\cal E}_2(w):=F_3(1,R_2^\star(1,w),R_3^\star(1,w),  w^k).\end{equation}
Finally, consider the previous equality in the affine chart,
\[{\cal E}_3(z):=F_3(z^{k+n},R_2^\star(z,1),z^n R_3^\star(z,1), z^n)=\]\[\Lambda_0(b_0)z^{\mu}+\Lambda_{1}(b_0, b_{1})z^{\mu-1}+\cdots+\Lambda_k(b_0, b_{1},\ldots,b_k) z^{\mu-k}+\sum_{j=k+1}^{\mu} \Lambda_j(b_0, b_{1},\ldots,b_k) z^{\mu-j}.  \]

Now, from equations \ref{E-uno} and \ref{E-dos}, we obtain the coefficients $\Lambda_i$, which construct the asymptote $\widetilde{\mathcal{Q}}$ introduced in the theorem statement. Specifically, equations \ref{E-uno} and \ref{E-dos} can be written as
\[{\cal E}_2(w)= \Lambda_0(b_0)+
\Lambda_1(b_0,b_1)w+\cdots+\Lambda_k(b_0, b_{1},\ldots,b_k) w^{k}+\cdots+ \Lambda_{\mu}(b_0, b_{1},\ldots,b_k) w^{\mu} \]
$$0={\cal E}_1(w)={\cal E}_2(w)+\Lambda_{\mu+1}(b_0, b_{1},\ldots,b_k) w^{\mu+1}+\cdots.$$
Thus,
\[\Lambda_0(b_0)=\Lambda_1(b_0,b_1)=\cdots=\Lambda_j(b_0,b_1,\ldots,b_k)=0,\quad j=1,\ldots,k. \]
Since $n\in {\Bbb N}$ is given by the maximum of $n_0$ (obtained in Lemma \ref{L-final}) and $k-\ell$ ($\ell$ is such that $a_{j}=0,\,\,j=0,\ldots,\ell$), we get   $k+1$ linearly independent equations defining a triangular system where the undetermined coefficients are $b_0,b_1,\ldots, b_k$.\\
 \end{proof}
\para

\begin{remark}\label{R-infpoint3}
\begin{itemize}
  \item[1.] Since the infinity points are of the form $(1:m_2:m_3:0)$, we get that $\deg(q_i)\leq k,\,i=2,3.$ Thus, once one get the asymptotes of the plane curve, we obtain the degree of the asymptotes of the space curves.
    \item[2.] In Theorem \ref{T-final}, we obtain a triangular system which is trivial to solve. Furthermore, we note that  $b_0=m_3$ is a root of $\Lambda_0(b_0)=0$ if and only if $(1:m_2: m_3:0)$ is an infinity point.
   \item[3.] The output asymptote could not be proper. In this case, we can reparametrize properly using for instance the method presented in \cite{Sonia}.\\
\end{itemize}
\end{remark}

\para

Generalizing the previous theorem, one may determine   as many terms as desired  for $r_3(z)$ by simply introducing enough terms of $r_2(z)$ into the equations obtained from $f_3(z, r_2(z), r_3(z)) = 0$. More precisely, reasoning similarly as in Theorem \ref{T-final} and using the notation previously introduced, the following corollary is obtained.

\para

\begin{corollary}\label{C-final}  Let  $B_p=\{(z,r_2(z))\in {\Bbb C}^2:\,z\in {\Bbb C},\,|z|>M\}$ be an infinity branch of the plane algebraic curve ${\cal C}_p$. The corresponding branch $B=\{(z,r_2(z),r_3(z))\in \mathbb{C}^2: \,z\in \mathbb{C},\,|z|>M\}$ of $\mathcal{C}$ is defined $$r_3(z^k)=b_0 z^{k}+b_{1}z^{k-1}+\cdots+b_{k}+b_{k+1}/z+b_{k+2}/z^2+\cdots+b_{k+m}/z^m,\,$$ where $b_i \in \mathbb{C}$ satisfy  that
$$\Lambda_j(b_0, b_{1},\ldots,b_j)=b_j \rho(a_0,\ldots,a_r)-\rho_j(a_0,\ldots,a_{r+j}),\qquad j=0,\ldots, m+k$$
 $\mu-m-k\geq 0$, $r+m\leq n$ and $m+\ell+1\geq k+1$, being
\[F_3(t^{n+k},a_0t^{n+k}+a_1t^{n+k-1}+\cdots+a_{n+k}, t^{n-m}(b_0 t^{k+m}+b_{1}t^{k+m-1}+\cdots+b_{k}t^m+b_{k+1}t^{m-1}+b_{m+k}),  t^n)=\]
$$\Lambda_0(b_0) t^{\mu}+\Lambda_{1}(b_0, b_{1}) t^{\mu-1}+\cdots+\Lambda_k(b_0, b_{1},\ldots,b_{m+k}) t^{\mu-m-k}+\sum_{j=m+k+1}^{\mu} \Lambda_j(b_0, b_{1},\ldots,b_{m+k}) t^{\mu-j}.$$
\end{corollary}

\para

Below, we introduce Algorithm {\sf Improvement Asymptotes Construction-Implicit Case}, which utilizes the above results to compute the g-asymptotes of a space curve. Furthermore, we illustrate the algorithm with two examples.
\para

\begin{center}
\fbox{\hspace*{2 mm}\parbox{15.5 cm}{ \vspace*{2 mm} {\bf Algorithm
{\sf Improvement Space Asymptotes Construction.}} \vspace*{0.2cm}

\noindent {\sf Given} an irreducible real algebraic space curve $\cal C$
implicitly defined
 by two polynomials $f_1(x_1,x_2,x_3),f_2(x_1,x_2,x_3)\in {\Bbb R}[x_1,x_2,x_3]$,  the algorithm {\sf outputs} an
asymptote for each of its infinity branches.

\begin{itemize}

\item[1.] Compute the projection of $\cal C$ along
the $x_3$-axis. Let  ${\cal C}_p$ be this projection and $f^p(x_1,x_2)$ the implicit polynomial defining ${\cal C}_p$.

\item[2.] Determine the lift function $h(x_1,x_2)$ (see
Sections 2 and 3 in \cite{Bajaj}). Let $f_3(x_1,x_2,x_3)=x_3h_2(x_1,x_2)-h_1(x_1,x_2)$.

\item[3.]  For each  infinity branch, $B_p=\{(z,r_2(z))\in {\Bbb C}^2:\,z\in {\Bbb C},\,|z|>M\}$,  of ${\cal C}_p$ determine the truncated branch (apply \cite{BlascoPerezII} or \cite{NewImpl-As})
$$r_2^\star(z)=a_0z+a_1z^{(k-1)/k}+a_2z^{(k-2)/k}+\cdots+a_k+a_{k+1}z^{-1/k}+\cdots+a_{k+n}z^{-n/k},$$
where  $n\in {\Bbb N}$ is given by the maximum of $n_0$ (obtained in Lemma \ref{L-final}; see also Remark \ref{R-final}) and $k-\ell$ ($\ell$ is  such that $a_{j}=0,\,\,j=0,\ldots,\ell$).

\item[4.] For each infinity branch  of ${\cal C}_p$, consider $\widetilde{\mathcal{Q}}(t)=(t^{k},\,q_2(t),\,  q_3(t)),\,$ $q_2(t)=a_0t^k+a_1t^{k-1}+\ldots+a_k,$ where $q_3(t)=b_0 t^{k}+b_{1}t^{k-1}+\cdots+b_{k},\,$ $b_i\in \mathbb{C},\,i=1,\ldots,k$ are undetermined coefficients.
\begin{itemize}
\item[4.1.] Compute
 \[F_3(t^{n+k},a_0t^{n+k}+a_1t^{n+k-1}+\cdots+a_{n+k}, t^n(b_0 t^{k}+b_{1}t^{k-1}+\cdots+b_{k}),  t^n)=\]
$$\Lambda_0(b_0) t^{\mu}+\Lambda_{1}(b_0, b_{1}) t^{\mu-1}+\cdots+\Lambda_k(b_0, b_{1},\ldots,b_k) t^{\mu-k}+\sum_{j=k+1}^{\mu} \Lambda_j(b_0, b_{1},\ldots,b_k) t^{\mu-j}$$

\item[4.2.] Solve the triangular system of equations $$\Lambda_i(b_0,\ldots, b_i)=0,\,\,i=0,\ldots,k$$ and substitute the solutions in $\widetilde{\mathcal{Q}}(t)$. Let $\widetilde{\mathcal{Q}}(t)$ be this  parametrization.

\item[4.3.] Return each asymptote $\widetilde{\cal C}$ defined by
$\widetilde{Q}(t)$.
\end{itemize}
\end{itemize}
}\hspace{2 mm}}
\end{center}

  \para

\begin{remark}\label{R-algoritmo2} We assume that the given algebraic space curve  $\cal C$  only has points at infinity of the form $(1 : m_2 : m_3 : 0)$; otherwise, a linear change of coordinates can be applied. However, we observe that if $(0 :1 : m_3 : 0)$ is a point at infinity, the only issue arises when computing the infinite branches of  ${\cal C}_p$, since  $(0:1:0)$  is   a point at infinity (see statement 2 in Remark \ref{R-algoritmo}). In this case, a linear change of coordinates can be applied to ${\cal C}_p$ to facilitate Step 3, and it can be reversed before proceeding to Step 4.

\end{remark}
    
  \para

\begin{example}\label{E-method-ant-a}
Consider the space curve ${\cal C}$ introduced in Example \ref{E-infbranches}. We apply Algorithm  {\sf Improvement Space Asymptotes Construction.} Steps 1 and 2 are easily applied as we showed in previous examples and we have that ${\cal C}_p$ is defined by the polynomial
\[f^p(x_1,x_2)=-x_1^2 x_2^2 + 2 x_1 x_2^3 - x_2^4 + x_1^2 x_2 - x_1 x_2^2 + 2 x_1 x_2 - x_2 + 2,\] and
\[f_3(x_1,x_2)=-x_1 x_2 + x_2^2 + x_3.\]
Now, we apply step 3 for each infinity branch  of ${\cal C}_p$. Using \cite{NewImpl-As}, we have that
\[r_{2,1}(z)=z-2+\cdots,\,\,r_{2,2}(z)=z+1+\cdots,\,\,r_{2,3}(z)=0+\cdots,\,\,r_{2,2}(z)=1+\cdots.\]
We apply Remark \ref{R-final} to determine  $r_{2,j}^\star(z),\,j=1,\ldots,4$. For $r_{2,1}^\star(z)$, we have that the leader coefficient of the numerator of the rational functions $f_3(t, t+v_1,b_0t+b1)$ and $f_3(t, t+v_1+v_2/t,b_0t+b1)$ depend  on $v_1$. Hence, we deduce that $r=1$ and we consider $n=1$ (we check that $\mu\geq k=1$ and $n\geq k-\ell=1$).
  Thus, let
  \[r_{2,1}^\star(z)=z-2-7/(3z)\]
 Reasoning as above, we get that
  \[r_{2,2}^\star(z)=z+1-2/(3z),\,\,r_{2,3}(z)^\star(z)=0+0/t,\,\,r_{2,4}(z)^\star(z)=1+3/z.\]
Now, we apply step 4 for each infinity branch  of  ${\cal C}_p$. Then,
\begin{itemize}
  \item  Let $\widetilde{\mathcal{Q}_1}(t)=(t, t-2,\,  b_0t+b1).$ We compute the numerator of
  $f_3(t, t-2-7/(3t), b_0 t+b_1)$ and get that
  \[\Lambda_0=b_0 - 2,\qquad \Lambda_1=b_1 + 5/3.\]
Hence $b_0=2$ and $b_1=-5/3$ and thus, $\widetilde{\mathcal{Q}_1}(t)=(t,  t-2,\,  2t-5/3).$

  \item Let $\widetilde{\mathcal{Q}_2}(t)=(t, t+1,\,  b_0t+b1).$  We compute the numerator of
  $f_3(t, t+1-2/(3t), b_0 t+b_1)$ and get that
  \[\Lambda_0=b_0 +1,\qquad \Lambda_1=b_1 + 1/3.\]
Hence  $b_0=-1$ and $b_1=-1/3$, and $\widetilde{\mathcal{Q}_2}(t)=(t, t+1,\, -t-1/3).$

  \item Let $\widetilde{\mathcal{Q}_3}(t)=(t, 0,\,  b_0t+b1).$ Reasoning as above, we compute the numerator of
  $f_3(t, 0, b_0 t+b_1)$. We get that
  \[\Lambda_0=b_0,\qquad \Lambda_1=b_1\]
which implies that $b_0=b_1=0$. Thus, $\widetilde{\mathcal{Q}_3}(t)=(t, 0,0).$

  \item Let $\widetilde{\mathcal{Q}_4}(t)=(t, 1,\,  b_0t+b1).$ We compute the numerator of
  $f_3(t, 1+3/t, b_0 t+b_1)$, and we get that
  \[\Lambda_0= b_0 - 1,\qquad \Lambda_1=b_1 - 2.\]
  Thus $b_0=1$ and $b_1=2$, and hence $\widetilde{\mathcal{Q}_4}(t)=(t, 1,\, t+2).$
\end{itemize}
The input curve, ${\cal C}$, and its four asymptotes have been plotted in Figure \ref{F-ejemplo-asintotas}. In Figure \ref{F-ejemplo1N}, one may see the input surfaces, together with the space curve and the asymptotes.In the right side, we are plotting only the space curve and the asymptote in a larger quadrant to observe the convergence as we approach infinity. \\

   \begin{figure}[h]
   \centering
   $\begin{array}{ccc}
    \includegraphics[width=0.54\textwidth]{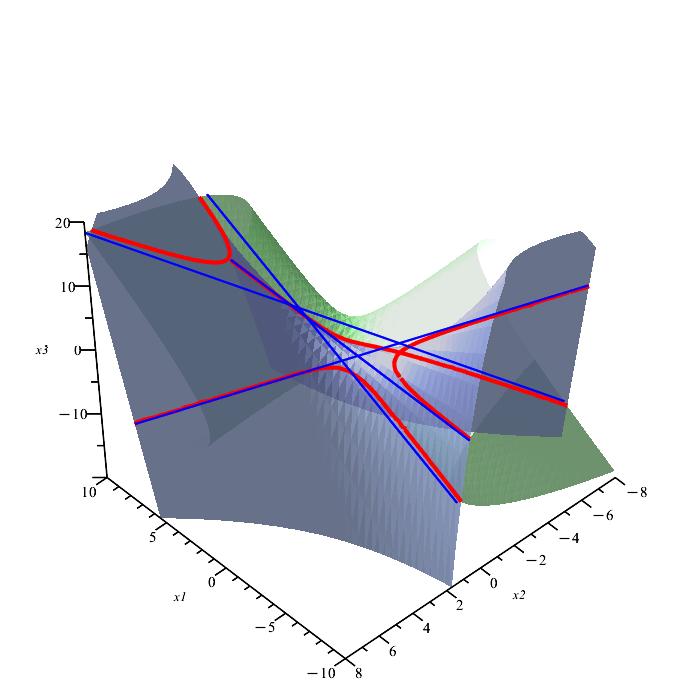}  &  \includegraphics[width=0.54\textwidth]{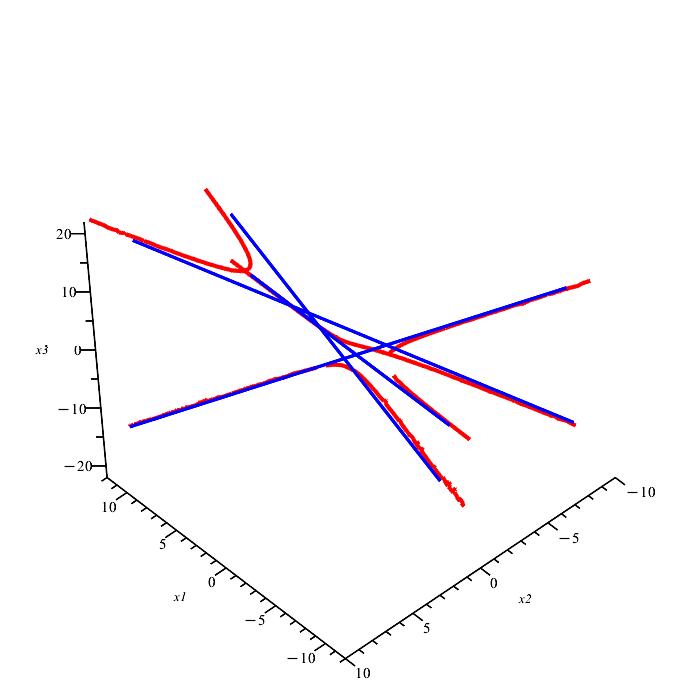}
   \end{array}$
  \caption{Input surfaces, space curve and asymptotes (left). Space curve and asymptotes (right).}\label{F-ejemplo1N}
 \end{figure}
\end{example}

\para

As we outlined in Remark \ref{R-final}, in the practice, in step 3 of the  algorithm  {\sf Improvement Space Asymptotes Construction},   it suffices to consider a sufficiently large  $n\in {\Bbb N}$ to obtain the correct result. Observe that calculating the truncated branch $r_2^{\star}(z)$ with more terms is not a problem at all, nor is it computationally expensive, because one only needs to solve easy triangular systems by applying the results from \cite{NewImpl-As}.\\

\para

\begin{example}\label{E-method-ant-b-new}
Consider the space curve ${\cal C}$ defined by the polynomials
\[f_1(x_1,x_2,x_3)=x_1 x_2^4 - x_2^5 - 2 x_1^2 x_2^2 + 4 x_1 x_2^3 - 2 x_2^4 + x_1^3 - 3 x_1^2 x_2 + 3 x_1 x_2^2 - x_2^3 - 4 x_1 x_2 + 4 x_2^2 - 1,\quad\]\[ f_2(x_1,x_2,x_3)=x_1^2 x_2 + 2 x_1 x_2 x_3 - x_2^2 x_3 + x_2^2 + x_1 - x_2 + x_3.\]We apply Algorithm  {\sf Improvement Space Asymptotes Construction.} Note that in this case, we have that  
$f^p(x_1,x_2)=f_1(x_1,x_2,x_3)$ is the implicit polynomial defining ${\cal C}_p$, and  $f_3(x_1,x_2)=f_2(x_1,x_2,x_3).$ \para

\noindent
Now, we apply step 3 for each infinity branch  of ${\cal C}_p$. Using \cite{NewImpl-As}, we have that
\[r_{2,1}(z)=z-1/z^4+\cdots,\,\,\,\,r_{2,2}(z)=-1/2+z^{1/2}+z^{-1/4}+1/8  z^{-1/2} +\cdots.\]
We apply Remark \ref{R-final} to determine  $r_{2,j}^\star(z),\,j=1,2$. For $r_{2,1}^\star(z)$, we have that the leader coefficient of numerator of $f_3(t, t+v_1,b_0t+b1)$ does not depend on $v_1$. Hence, we deduce that $r=0$ and we consider $n=0$ (we check that $\mu\geq k=1$ and $n\geq k-\ell=0$).
  Thus, let
  \[r_{2,1}^\star(z)=z\]
For $r_{2,2}^\star(z)$, we note that $\ell=2$ and thus $r\geq 2$.  In addition, the leader coefficient of the numerator of $f_3(t^4, t^2 - 1/2+1/t+1/(8t^2)+v_1/t^3,b_0t^4+b_1t^3+b_2t^2+b_3t+b_4)$ does not depend on $v_1$, and thus $r=2$. We consider $n=2$ (we check that $\mu\geq k=4$ and $n\geq k-\ell=2$), and we have that
  \[r_{2,2}^\star(z)=-1/2+z^{1/2}.\]
Now, we apply step 4 for each infinity branch  of  ${\cal C}_p$. Then,

\begin{itemize}
  \item  Let $\widetilde{\mathcal{Q}_1}(t)=(t, t,\,  b_0t+b_1).$  We compute the numerator of
  $f_3(t, t, b_0 t+b_1)$, and we get that
    \[\Lambda_0=b_0+1,\qquad \Lambda_1=b_1+1.\]
Thus $b_0=b_1=0$, and $\widetilde{\mathcal{Q}_1}(t)=(t,  t,\,  0).$
  \item Let $\widetilde{\mathcal{Q}_2}(t)=(t^4,\,   t^2 - 1/2,\,  b_0t^4+b_1t^3+b_2t^2+b_3t+b_4).$  We compute the numerator of
   $f_3(t^4, t^2 - 1/2+1/t+1/(8t^2),b_0t^4+b_1t^3+b_2t^2+b_3t+b_4)$, and we get that
    \[\Lambda_0=8 b_0 + 4,\quad \Lambda_1=8 b_1,\quad\Lambda_2=8 b_2 - 8 b_0 - 2,\quad \Lambda_3=-8 b_1 + 8 b_2,\quad\Lambda_4=-8b_2 + 8b_43 + 4b_0.\]
 Hence  $b_0=-1/2, b_1=0, b_2 =-1/4, b_3 =0, b_4 =0$, and then, $\widetilde{\mathcal{Q}_2}(t)=(t^4,\,   t^2 - 1/2,\, -1/2 t^4-1/4t^2).$ Note that in this case $\widetilde{\mathcal{Q}_2}$ is not proper but it can be properly reparametrized as $(t^2,\,   t - 1/2,\, -1/2 t^2-1/4t)$ (see statement 3 in Remark \ref{R-infpoint3}).
\end{itemize}

In Figure \ref{F-ejemplo2N}, one may see the input surfaces, together with the space curve and the asymptotes. In the right side, we are plotting only the space curve and the asymptote in a larger quadrant to observe the convergence as we approach infinity. \\

   \begin{figure}[h]
   \centering
   $\begin{array}{ccc}
    \includegraphics[width=0.54\textwidth]{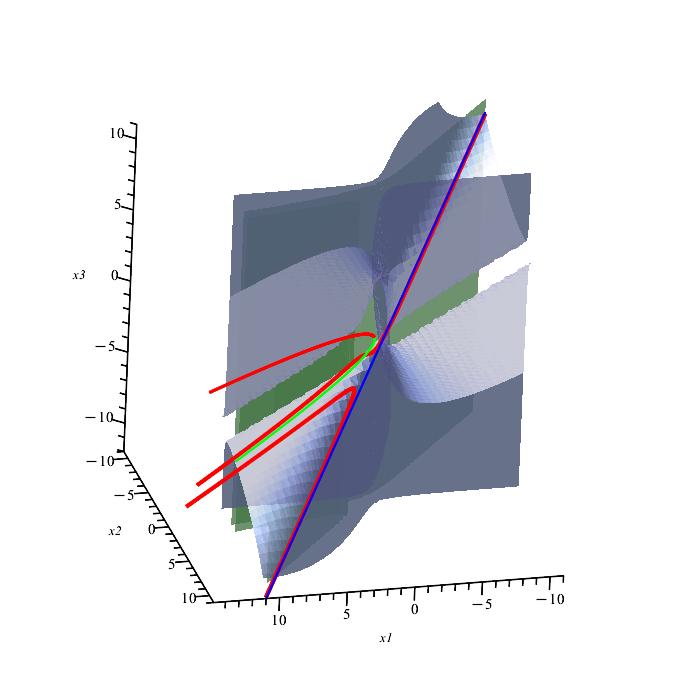}  &  \includegraphics[width=0.54\textwidth]{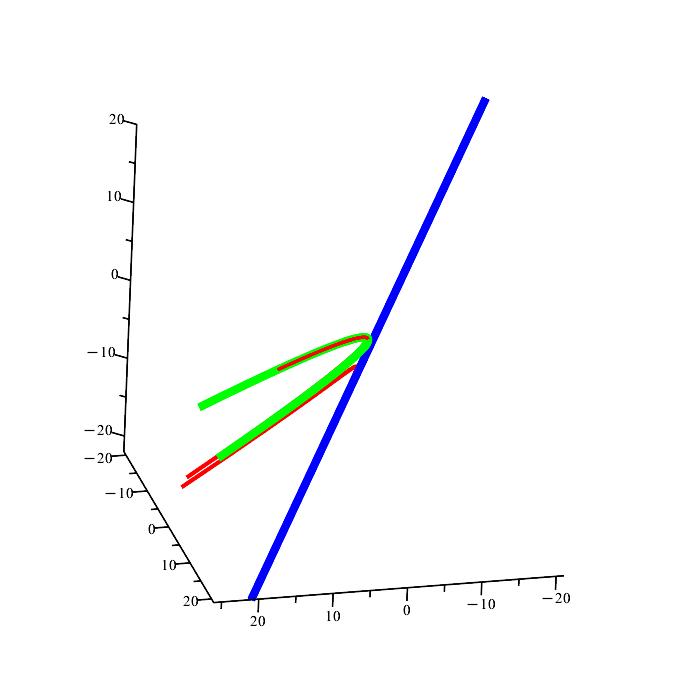}
   \end{array}$
  \caption{Input surfaces, space curve and asymptotes (left). Space curve and asymptotes (right).}\label{F-ejemplo2N}
 \end{figure}
\end{example}
 
\para

The method described above may be trivially adapted for dealing with algebraic curves $\cal C$ in the $n$-dimensional space defined by irreducible polynomials $f_{j}(x_1,x_2,\ldots, x_n),\,j=1,\ldots,n-1$. For instance, if $n=4$, and we have a curve $\cal C$ defined by the irreducible polynomials $f_j(x_1,x_2,x_3,x_4),\,j=1,2,3$, we compute the asymptotes of the  curves defined by the projection over a valid projection direction (under a linear change of coordinates, one may assume that  the $x_n$-axis serves as a valid projection direction). Afterwards, we  consider the corresponding lift-function to obtain the asymptotes of the input curve.

\para

Algorithms  \textsf{Asymptotes Construction-Implicit Case} and \textsf{Improvement Asymptotes Construction-Implicit Case}  allow  us to easily obtain all the generalized asymptotes of an algebraic  curve implicitly defined. However,  one has to determine the roots of some given equations, which may entail certain difficulties if algebraic numbers are involved. For this purpose, we use the notion of  \textit{conjugate points} (see  \cite{SWP}), which will help us to overcome this problem. The idea is to
collect points whose coordinates depend algebraically on all conjugate roots of the same
 irreducible polynomial, say $m(t)\in \mathbb{R}[t]$. This will imply that the computations on such families can be carried
out by using the defining polynomial $m(t)$ of these algebraic numbers. That is, one applies the formulae presented in Theorem \ref{T-final}, but  modulo $m(t)$, i.e.  we use the polynomial  $m(t)$ to carry out the arithmetic by computing polynomial remainders (see e.g. \cite{NewImpl-As} or \cite{NewPara-As}).  We should note that, according to Corollary \ref{C-final}, the presence of algebraic numbers occurs because they have already appeared during the computation of  $r_2$. In the following, we illustrate this process with a simple example that clearly conveys the idea.

\para

\begin{example}\label{E-conjugate}
Consider the space curve ${\cal C}$ defined by the polynomials
\[f_1(x_1,x_2,x_3)= 2x_1^3 + x_1x_3^2 + x_3^3 + 4x_31,\quad f_2(x_1,x_2,x_3)=-x_1^2 - x_3^2 + x_2.\]
We apply Algorithm  {\sf Improvement Space Asymptotes Construction.} In this case, we have that  
$$f^p(x_1,x_2)=-2x_1^6 + x_1^4x_2 + 8x_1^4 - 4x_1^2x_2^2 - 16x_1^2x_2 + x_2^3 - 16x_1^2 + 8x_2^2 + 16x_2$$ is the implicit polynomial defining ${\cal C}_p$, and  $$f_3(x_1,x_2)=(-x_1^2 + x_2 + 4)x_3 + x_1^3 + x_1x_2.$$

\noindent
Now, we apply step 3 for each infinity branch  of ${\cal C}_p$. Using \cite{NewImpl-As}, we have that
\[r_{2}(z)= z^2\lambda + 4\lambda^2/29 - 36\lambda/29 - 48/29 +(384\lambda^2/841  - 1368 \lambda/841  + 32/841)z^{-2} +\cdots,\]
where $m(\lambda)=0$, and $m$ is the irreducible polynomial $m(t)=t^3 - 4t^2 +t - 2$ (we had to consider a change of coordinates because the point at infinity of the projective plane curve is (0:1:0); see Remark \ref{R-algoritmo2}).  We  determine  $r_{2}^\star(z)$ reasoning as in the previous examples  to get that 
  \[r_{2}^\star(z)=z^2\lambda + 4\lambda^2/29 - 36\lambda/29 - 48/29 \]
and now,  we apply step 4. Then, let $\widetilde{\mathcal{Q}}(t)=(t,\,  t^2\lambda + 4\lambda^2/29 - 36\lambda/29 - 48/29,\,  b_0t^2+b_1t+b_2).$  We compute the numerator of
   $f_3(t,t^2\lambda + 4\lambda^2/29 - 36\lambda/29 - 48/29,  b_0t^2+b_1t+b_2)$, and we get that
    \[\Lambda_0=29 b_0\lambda - 29b_0,\quad \Lambda_1=29b_1\lambda - 29b_1 + 29\lambda + 29,\quad\Lambda_2=4b_0\lambda^2 - 36b_0\lambda + 29b_2\lambda + 68b_0 - 29b_2.\]
 Hence  $b_0=0, b_1=-(\lambda + 1)/(\lambda - 1), b_2=0$, and then  $$\widetilde{\mathcal{Q}}(t)=(t,\,   t^2\lambda + 4\lambda^2/29 - 36\lambda/29 - 48/29,\,-(\lambda + 1)/(\lambda - 1)t)$$
 where $m(\lambda)=0$, and $m(t)=t^3 - 4t^2 +t - 2$. Applying resultants,   and we get that this asymptote is defined by $g_1(x_1,x_2,x_3)=g_2(x_1,x_2,x_3)=0$, where 
 \[g_1=48778x_1^6 - 24389x_1^4x_2 - 195112x_1^4 + 97556x_1^2x_2^2 + 390224x_1^2x_2 - 24389x_2^3 + 538240x_1^2 - 195112x_2^2 - 484416x_2 - 430592,\]
 \[g_2=8x_1^3 + 4x_1x_3^2 + 4x_3^3.\]
In Figure \ref{F-ejemplo3N}, one may see the input surfaces, together with the space curve and the asymptotes. In the right side, we are plotting only the space curve and the asymptote in a larger quadrant to observe the convergence as we approach infinity. \\

   \begin{figure}[h]
   \centering
   $\begin{array}{ccc}
    \includegraphics[width=0.54\textwidth]{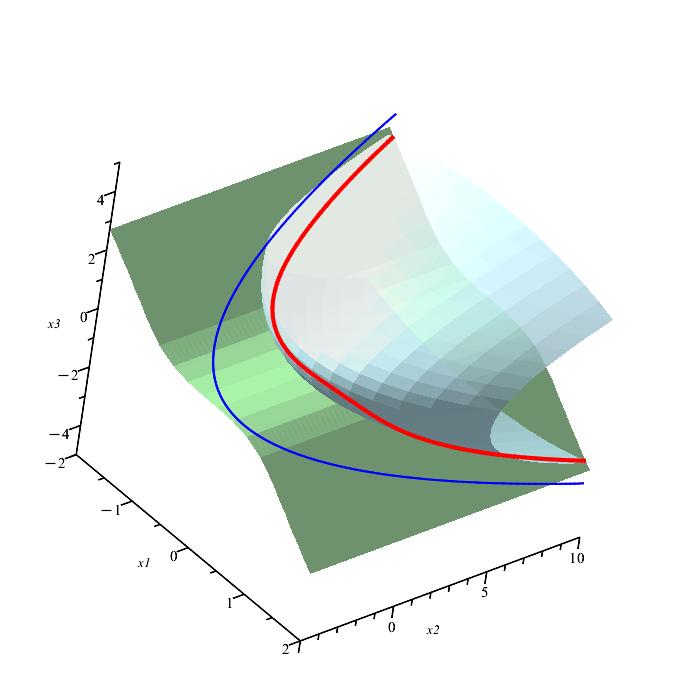}  &  \includegraphics[width=0.54\textwidth]{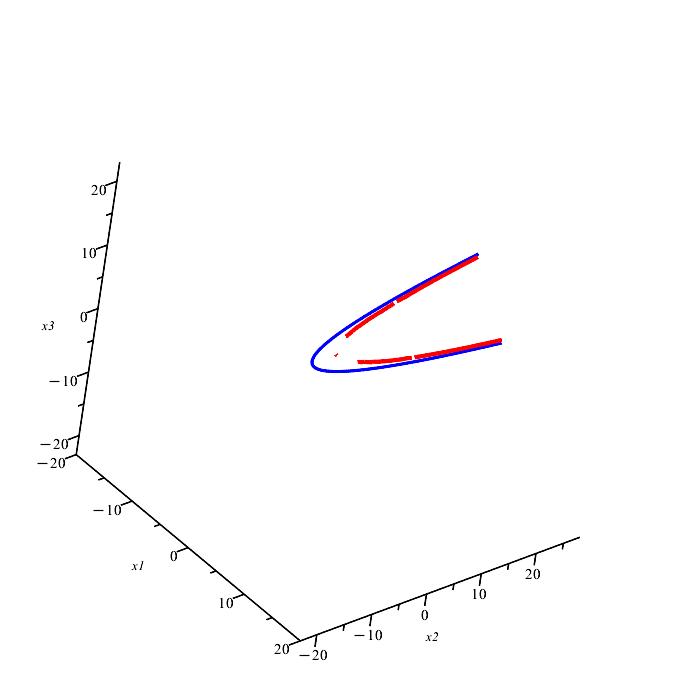}
   \end{array}$
  \caption{Input surfaces, space curve and asymptotes (left). Space curve and asymptotes (right).}\label{F-ejemplo3N}
 \end{figure}
\end{example}

\section{Concluding Remarks}\label{Sect4}

In this manuscript, we introduce essential tools for analyzing the asymptotic behavior of real algebraic space curves implicitly  defined. Specifically, we remind the concepts of "infinity branches" and "generalized asymptotes," explore their properties, and outline algorithms for their computation. While these notions were previously introduced for implicit real algebraic plane curves, their treatment in the context of space curves necessitates distinct approaches. Our contributions in this paper include:
\begin{enumerate}
\item Introduction of fundamental notions and results tailored to algebraic space curves, including the definitions of infinity branches and approaching curves. These concepts represent direct extensions from those established for plane curves (cf. Sections 3 and 4 in \cite{paper1}).
\item Presentation of a method for computing infinity branches in space, achieved by leveraging reductions from the spatial domain to the plane. This approach   requires repetitive computation of the entire infinite branch using Puiseux series.
\item Development of a new procedure for computing asymptotes of implicitly defined space curves. More precisely, we characterize the existence of the asymptotes   of $\mathcal{C}$ from the asymptotes of ${\mathcal{C}^p}$ and  we provide an effective algorithm for computing the asymptotes of $\mathcal{C}$ once the asymptotes of ${\mathcal{C}^p}$ are determined. We do not need to compute infinity branches and Puiseux series and instead, we only need to identify the solutions of a triangular system of equations derived from the implicit polynomial.

  The novel approach could be applicable to algebraic curves in higher-dimensional spaces.

\end{enumerate}
Our future work is based on three main points. On one hand, although some previous results have been presented in \cite{Surf-As}, there is still a need for a deep analysis for the characterization and computation of asymptotes in the case of surfaces. Therefore, we aim to extend or generalize the results to the case of surfaces. On the other hand, the generation of families of algebraic curves with given asymptotes. In this regard, it is important to note the insight that can be shed on the problem of determining varieties resulting from the intersection of two surfaces. Finally, we note that the case of space curves defined by more than two polynomials can also be addressed using the techniques employed in this work. However, it requires a more in-depth study and analysis, including the generalization of the construction of a proper lift function when the curve is not derived from a complete intersection  (see \cite{Bajaj}).

\section{Acknowledgements}

First  author is partially supported by  Ministerio de Ciencia, Innovaci\'on y Universidades - Agencia Estatal de Investigaci\'on/PID2020-113192GB-I00 (Mathematical Visualization: Foundations, Algorithms and Applications).
First  author  belongs to the Research Group ASYNACS (Ref.CCEE2011/R34). Second author is partially supported by National Natural Science Foundation of China under grant 12371384 and Fundamental Research Funds for the Central Universities.
The author R. Magdalena Benedicto is partially supported by the State Plan for Scientific and Technical Research and Innovation of the Spanish MCI (PID2021-127946OB-I00). \\

\end{document}